\documentclass[a4paper,reqno,oneside,10pt]{amsart} 

\usepackage[frenchb]{babel}
\usepackage[latin1]{inputenc}
\usepackage{epsfig}
\usepackage{graphicx}
\usepackage{psfrag}

\usepackage{mathrsfs}
\usepackage{amsmath}
\usepackage{amssymb}
\usepackage{extarrows}
\usepackage{enumerate}

\newcommand{\C}{\mathbf{C}}
\newcommand{\R}{\mathbf{R}}
\newcommand{\Q}{\mathbf{Q}}
\newcommand{\Z}{\mathbf{Z}}

\newcommand{\Eu}{\mathbf{E}}
\newcommand{\Sph}{\mathbf{S}}

\newcommand{\Hyp}{\mathbf{H}}

\def\a{\alpha}
\def\b{\beta}
\def\g{\gamma}

\def\e{\varepsilon}
\def\l{\lambda}

\def\L{\Lambda}
\def\G{\Gamma}

\newcommand{\cusp}{\mathcal{C}}
\newcommand{\sys}{\mathsf{sys}}
\newcommand{\vol}{\mathsf{vol}}

\newcommand{\covol}{\mathsf{covol}}

\newcommand{\rr}{\mathsf{R}}
\newcommand{\re}{\mathsf{r}}
\newcommand{\Mob}{\mathsf{Mob}}

\newcommand{\Isom}{\mathsf{Isom}}
\newcommand{\PSL}{\mathsf{PSL}}
\newcommand{\SL}{\mathsf{SL}}
\newcommand{\PGL}{\mathsf{PGL}}

\newcommand{\Kl}{\mathbf{Kl}}

\newcommand{\ic}{\mathsf{i}_{\mathcal{C}}}
\newcommand{\m}{\mathsf{m}}
\newcommand{\Tr}{\mathsf{Tr}}

\newcommand{\Real}{\mathsf{Re}}
\newcommand{\Ima}{\mathsf{Im}}

\theoremstyle{plain}
\newtheorem{theorem}{Th\'eor\`eme}[section] 
\newtheorem*{theoremnonumber}{Th\'eor\`eme}

\newtheorem{proposition}[theorem]{Proposition}
\newtheorem*{propositionnonumber}{Proposition}
\newtheorem{lemma}[theorem]{Lemme}
\newtheorem*{lemmanonumber}{Lemme}
\newtheorem{thm-defi}[theorem]{Théorème - Définition}
\newtheorem{pro-defi}[theorem]{Proposition - Définition}

\theoremstyle{definition}

\newtheorem{remark}{Remarque}[section]
\newtheorem*{remarknonumber}{Remarque}

\title[Systole et rayon maximal]{Systole et rayon maximal des variétés hyperboliques non compactes}
\subjclass[2000]{57M50\and 30F45}
\keywords{Hyperbolic Manifolds, Cusps, Systole\and Inradius}
\date{Le \today}
\thanks{The author has been supported by the Swiss National Science Foundation through SNF projects number $200020-121506/1$ and $200021-131967/1$ supervised by Prof. R.~Kellerhals.}

\begin{document}

\renewcommand{\contentsname}{Plan de l'article} 

\maketitle

\begin{flushleft} 
               \textbf{Matthieu Gendulphe}\\
  \begin{small}Département de Mathématiques, université de Fribourg\\
               chemin du Musée 23, 1700 Fribourg Pérolles, Suisse\\
               courriel : Matthieu.Gendulphe@unifr.ch\end{small}
\end{flushleft}

\begin{abstract}
Nous contrôlons deux invariants globaux des variétés hyperboliques à bouts cuspidaux~: la longueur de la plus courte géodésique fermée (la \emph{systole}), et le rayon de la plus grande boule plongée (le \emph{rayon maximal}). Nous majorons la systole en fonction de la dimension et du volume simplicial. Nous minorons le rayon maximal par une constante positive indépendante de la dimension. Ces bornes sont optimales en dimension $3$. Cela donne une nouvelle caractérisation de la variété de Gieseking.
\end{abstract}

\renewcommand{\abstractname}{Abstract} 
\begin{abstract}
We bound two global invariants of cusped hyperbolic manifolds: the length of the shortest closed geodesic (the \emph{systole}), and the radius of the biggest embedded ball (the \emph{inradius}). We give an upper bound for the systole, expressed in terms of the dimension and simplicial volume. We find a positive lower bound on the inradius independent of the dimension. These bounds are sharp in dimension $3$, realized by the Gieseking manifold. It provides a new characterization of this manifold.
\end{abstract}
\vspace{1cm}

\setcounter{tocdepth}{1}   
\tableofcontents

 \newpage
\section{Introduction}\label{sec:intro}

 La recherche de bornes optimales sur les invariants globaux des variétés hyperboliques reste largement ouverte en dimension $n\geq 3$. Dans cet article, nous établissons des inégalités sur la systole et le rayon maximal des variétés hyperboliques non compactes. Ces inégalités sont optimales en dimension $3$.\par
Précisons tout de suite qu'une variété hyperbolique $M^n$ est une variété sans bord, munie d'une métrique complète de courbure sectionnelle constante $-1$. Sauf mention du contraire, une variété hyperbolique est supposée de volume fini.

\subsection{\'Enoncé des résultats}

La \emph{systole} d'une variété hyperbolique est la longueur de sa plus courte géodésique fermée, on la note $\sys$. Comme pour tout invariant métrique, nous souhaitons contrôler la systole en fonction de quantités purement topologiques, tel le volume simplicial noté $\vol_\triangle$. Nous considérons ici le cas des variétés hyperboliques non compactes. En dimension $3$, nous obtenons une inégalité optimale~:
\begin{theorem}\label{theo:systole}
Soit $M$ une $3$-variété hyperbolique non compacte, alors
\begin{eqnarray*}
\cosh(\sys(M)/2) & \leq & \frac{1+\sqrt{13}}{4}\cdot \vol_\triangle(M),
\end{eqnarray*}
avec égalité si et seulement si $M$ est isométrique à la variété de Gieseking.
\end{theorem}
\noindent En dimension $n\geq 4$ nous déterminons une constante positive $c_n$ telle que
\begin{eqnarray*}
\cosh(\sys/2) & \leq & c_n\cdot \vol_\triangle
\end{eqnarray*}
sur l'ensemble des $n$-variétés hyperboliques non compactes. L'expression de $c_n$ fait intervenir deux invariants des variétés plates~: la constante d'Hermite $\g_{n-1}$, et le plus grand indice $\mathsf{i}_{n-1}$ du sous-groupe abélien libre maximal d'un groupe de Bieberbach de l'espace euclidien $\Eu^{n-1}$. Une bonne estimation de la constante d'Hermite donne~:
\begin{eqnarray*}
c_n & \lesssim& \frac{\mathsf{i}_{n-1}}{5^{n-1}}.
\end{eqnarray*}
Nous ne disposons pas à l'heure actuelle de bornes précises sur $\mathsf{i}_{n-1}$.\par
  Le \emph{rayon maximal} d'une variété hyperbolique est le rayon de la plus grande boule métrique plongée dans la variété, on le note $\rr$. Nous minorons le rayon maximal par une constante positive indépendante de la dimension~:
\begin{theorem}\label{theo:R}
Soit $M$ une variété hyperbolique non compacte, alors
\begin{eqnarray*}
\cosh(\rr(M)) & \geq &  \frac{\sqrt{5}}{2}.
\end{eqnarray*}
La variété de Gieseking réalise l'égalité en dimension $3$.
\end{theorem}

\subsection{Petit panorama des résultats connus}
Nous effectuons ci-dessous un survol de la systole et du rayon maximal des variétés hyperboliques. Au passage, nous décrivons une méthode de majoration de la systole introduite par C.~Adams et A.~Reid dans \cite{adams-reid}. Nous généraliserons cette méthode à la partie~\ref{sec:systole-asymptotique}. 

\subsubsection{Majorations}
Soit $M$ une variété hyperbolique. En comparant le volume d'une boule de rayon $\rr(M)$ avec le volume de $M$ on trouve 
\begin{eqnarray*}
\cosh(\rr(M)) & \leq & \mathrm{const}_n\cdot \vol_\triangle(M).
\end{eqnarray*}
Si la variété est fermée, la demi-systole s'identifie au minimum du rayon d'injectivité, et on peut remplacer $\rr(M)$ par $\sys(M)/2$ dans la majoration précédente. Cet argument ne fonctionne pas en général, la variété de Gieseking par exemple vérifie $\sys(M)/2> \rr(M)$.\par
 Si $M$ admet une cuspide, on travaille avec la fibre de plus grand volume de la cuspide. Cette fibre consiste en une variété plate $N$, immergée mais non plongée dans $M$. Comme des points de $N$ sont identifiés dans $M$, on peut concaténer dans $M$ des lacets de $N$ ayant des points base distincts. Prenons le cas où $N$ est un tore plat, en choisissant deux lacets disjoints réalisant la systole de $N$, on construit un lacet de $M$ de longueur $2\sys(N)$. Les figures~\ref{fig:intro} et \ref{fig:noeud} illustrent cette manipulation.\par 
\begin{figure}[h]
\centering
\psfrag{N}{$N$}\psfrag{M}{$M$}
\begin{minipage}[b]{.48\linewidth}
 \centering\epsfig{figure=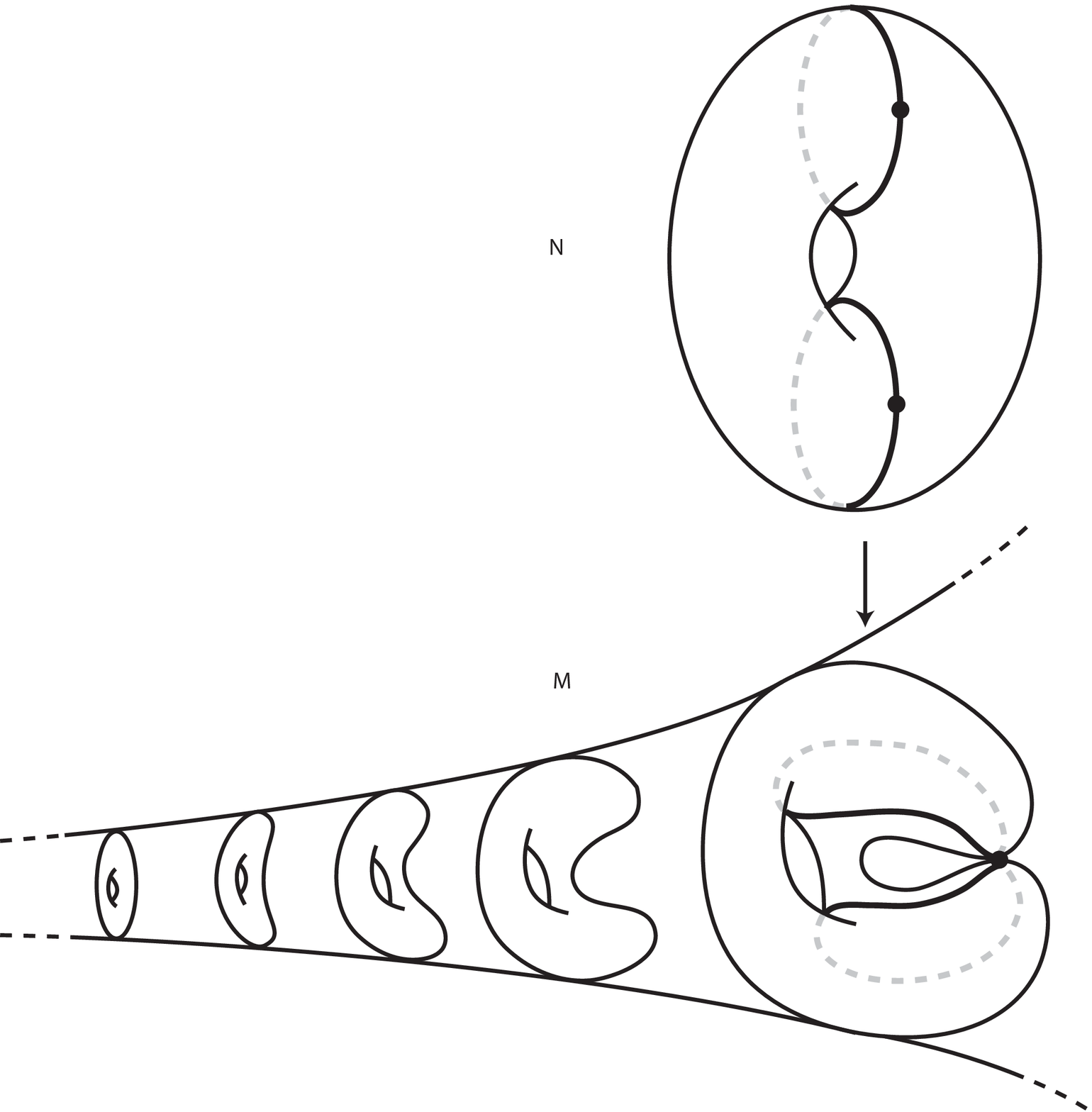,totalheight=6cm}
 \caption{Fibre maximale}\label{fig:intro}
\end{minipage}
\begin{minipage}[b]{.48\linewidth}
 \centering\epsfig{figure=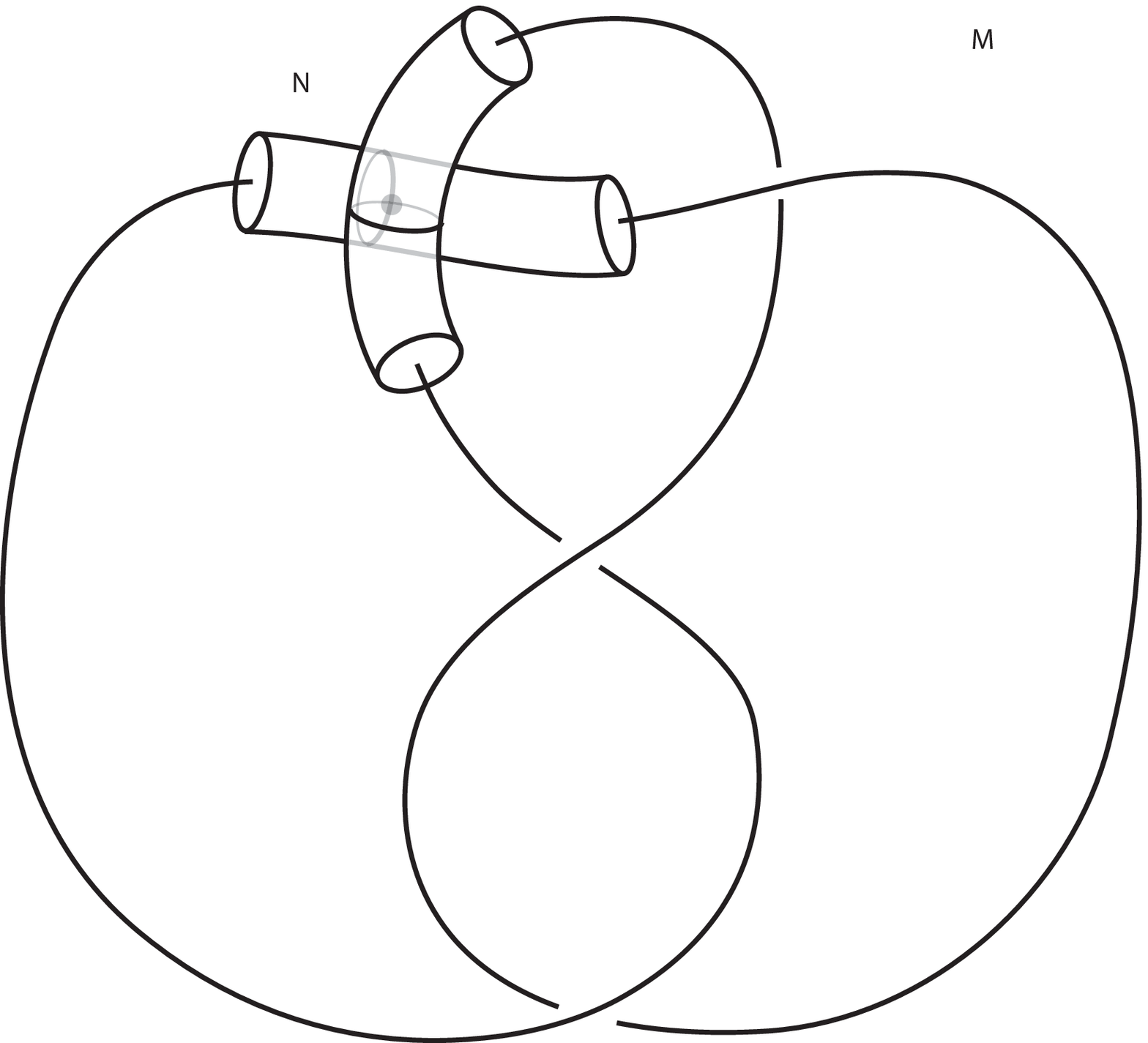,totalheight=6cm}
 \caption{N\oe ud de huit}\label{fig:noeud}
\end{minipage}
\end{figure}
  À ce moment précis, on ne sait pas si le lacet obtenu est périphérique. Il s'avère très difficile de répondre à cette question en toute généralité (voir \cite{wada}). Mais en dimension $3$, un petit calcul dans $\PSL(2,\C)$ fait apparaître que si le lacet obtenu est périphérique, alors $M$ contient un pantalon à trois pointes immergé. Ceci conduit à l'alternative $\sys(M)\leq 2\sys(N)$ ou $\sys(M)\leq 2\mathrm{arccosh}(3)$.\par
 L'idée de concaténer deux petits lacets périphériques est apparue dans l'article \cite{adams-reid} de C.~Adams et A.~Reid. Leur inspiration vient d'une part des travaux sur le volume minimal des variétés hyperboliques non compactes (voir la partie~\ref{sec:volume}), et d'autre part du problème des remplissages de Dehn exceptionnels. Voici un de leurs résultats~:
   
\begin{theoremnonumber}[C.~Adams, A.~Reid]
Soit $K$ un complément d'entrelacs hyperbolique, alors
\begin{eqnarray*}
\sys(K) & \leq & 4\pi.
\end{eqnarray*}
\end{theoremnonumber}

\begin{proof}
On reprend l'alternative ci-dessus, en utilisant la majoration $\sys(N)\leq2\pi$ provenant du $2\pi$-théorème de M.~Gromov et W.~Thurston.
\end{proof}

 On termine avec deux résultats de non majoration. Le premier dit que le rayon maximal n'admet pas de borne supérieure sur l'ensemble des compléments de n\oe uds hyperboliques (J.~Purcell et J.~Souto \cite{purcell}), ceci contraste avec le théorème précédent. Le second paraît intuitivement évident~:

\begin{propositionnonumber}[folklore]
À revêtement fini près, la systole et le rayon maximal d'une variété hyperbolique sont arbitrairement grands.
\end{propositionnonumber}

\begin{remarknonumber}
 M.~Gromov affirme dans \cite{gromov} \textsection~$0.2$ qu'une variété hyperbolique fermée admet une suite de revêtements finis $(M_i)_{i}$ de degrés $(d_i)_i$ telle que $\sys(M_i)$ soit comparable à $\log d_i$ quand $d_i$ tend vers l'infini.
\end{remarknonumber}

\begin{proof}
Soient $M$ une variété hyperbolique, et $\G$ un groupe uniformisant $M$. On se donne une constante $A>0$, et un point $x_0$ dans le revêtement universel. Par finitude géométrique de $\G$, le sous-ensemble
$$\G(x_0,A)=\left\{\g\in\G\ \mathrm{satisfaisant}\ 0<\ell_\g<A\ \mathrm{ou}\ d(x_0,\g x_0)<3A \right\}$$
est fini modulo conjugaison dans $\G$. Ici $\ell_\g$ désigne la distance de translation de $\g$ dans le revêtement universel. Le groupe $\G$ étant résiduellement fini (théorème de Mal'cev, voir \cite{magnus}), il admet un sous-groupe distingué d'indice fini et d'intersection vide avec $\G(x_0,A)$. Le revêtement fini associé vérifie $\sys\geq A$ et $\rr\geq A$.
\end{proof}

\subsubsection{Minorations}
 On attribue généralement à G.~Margulis et D.~Kazdan l'existence d'une constante positive $\rr_n$ minorant le rayon maximal des variétés hyperboliques de dimension $n$. Aucune constante $\rr_n$ n'est connue à l'heure actuelle (on parlera un peu plus loin du résultat de A.~Yamada). Plusieurs auteurs ont cependant déterminé des bornes explicites en utilisant des généralisations de l'inégalité de J\o rgensen ou du lemme du collier (voir le tableau~\ref{tab}). Ces méthodes permettent aussi de comparer le rayon maximal avec le volume ou le diamètre (A.~Reznikov \cite{reznikov}, R.~Kellerhals \cite{kellerhals01,kellerhals03,kellerhals04}, P.~Buser et H.~Karcher \cite{buser}). La minoration de P.~Buser et H.~Karcher reste valable pour les variétés riemanniennes fermées avec courbure pincée $-1\leq K<0$.
 
\begin{table}[h]
\centering
\renewcommand{\arraystretch}{1.5}
\begin{tabular}{|l|l|}
\hline
Minoration & Référence \\
\hline\hline
 $\rr(M^2)\geq \mathrm{arcsinh}(2/\sqrt{3})$\quad($M^2$ orientable) & A.~Yamada \cite{yamada} \\ 
\hline
 $\rr(M^3) \geq 0.04$\quad($M^3$ orientable) & P. L.~Waterman \cite{waterman} \\
 $\rr(M^3)\geq 0.24$\quad ($M^3$ fermée orientable) & A.~Prezworski \cite{prezworski} \textsection~$4$ \\ 
 \hline
 $\rr(M^5)> 0.033\ldots$\quad($M^5$ fermée orientable) & R.~Kellerhals \cite{kellerhals03} \textsection~$3$\\ 
 \hline
 $\rr(M^n) \geq 1/4^{n+3}$\quad ($M^n$ fermée) & P.~Buser et H.~Karcher \cite{buser} \textsection~$2.5$ \\
 $\rr_n\geq 0.0025/17^{[n/2]}$ & G.~Martin \cite{martin}, S.~Friedland et\\
où [n/2] désigne la partie entière de $n/2$  & S.~Hersonsky \cite{friedland} \textsection~$4$\\
 $\rr(M^n) >1/((n+3)\cdot \pi^{n-1})$\quad ($M^n$ orientable) & R.~Kellerhals \cite{kellerhals04} \\ 
 \hline
\end{tabular}
\medskip
\caption{Minorations de $\rr_n$}\label{tab}
\end{table}
On passe maintenant à la systole. Par remplissage de Dehn (voir \cite{thurston} \textsection~$5.8$ ou \cite{marden} \textsection~$4.11$) on construit facilement des $3$-variétés hyperboliques avec une systole arbitrairement petite. Ces variétés peuvent être compactes, non compactes, voire des compléments de n\oe ud (C.~Adams \cite{adams05} \textsection~$6$). Suivant une idée similaire (beaucoup plus difficile à mettre en \oe uvre) F.~Bonahon et J.-P.~Otal ont produit dans  \cite{bonahon} des variétés hyperboliques de volume infini avec une systole nulle (on définit alors la systole comme la borne inférieure des longueurs des géodésiques fermées). En dimensions supérieures, on obtient des variétés hyperboliques (compactes ou non compactes) avec une systole arbitrairement petite \emph{via} une manipulation algébrique introduite par I.~ Agol (\cite{agol}) pour la dimension $4$, et étendue en toute dimension par M.~Belolipetsky et S.~Thomson (\cite{belolipetsky}).

\subsubsection{Quelques mots sur les surfaces} L'étude des invariants globaux est nettement plus avancée en dimension $2$. Une borne inférieure optimale sur $\rr$ a été déterminée par A.~Yamada (\cite{yamada}) dans le cas orientable. Elle est réalisée par le pantalon à trois pointes, et atteinte à l'infini dans l'espace des modules pour les autres types topologiques (il suffit de pincer des géodésiques). Le maximum de $\rr$ sur l'espace des modules est connu dans le cas compact (C.~Bavard \cite{bavard96}). Signalons que des problèmes plus fins tels que la caractérisation des surfaces maximiales (C.~Bavard \cite{bavard96}), leur dénombrement à topologie fixée (\cite{bacher}), ou l'unicité du plus grand disque plongé dans ces surfaces (E.~Girondo et G.~Gonz{\'a}lez-Diez \cite{girondo}) ont été résolus.\par
 L'étude de la systole s'avère beaucoup plus difficile. Bien sûr, la systole n'admet pas de borne inférieure positive (en dehors du cas trivial du pantalon), il s'agit donc d'étudier son maximum sur l'espace des modules. La méthode de comparaison des volumes évoquée plus haut donne $\cosh(\sys/2)\leq \frac{1}{2}\cdot \vol_\triangle+1$ dans le cas fermé. Les exemples les plus intéressants de surfaces avec une grande systole (P.~Buser et P.~Sarnak \cite{sarnak}) ne permettent pas de dire si cette inégalité est asymptotiquement optimale. Nous connaissons peu de maxima globaux dans le cas compact (F.~Jenni \cite{jenni} et C.~Bavard \cite{bavard92}, P.~Schmutz Schaller \cite{schmutz93}, l'auteur \cite{gendulphe}), et une infinité dans le cas non compact (ils sont réalisés par les principaux sous-groupes de congruence, P.~Schmutz Schaller \cite{schmutz94}). Enfin, signalons ce qui est peut-être le plus intéressant, à savoir qu'une très belle théorie variationnelle a été développée (P.~Schmutz Schaller \cite{schmutz93}, C.~Bavard \cite{bavard97}).

\subsection{Commentaires sur les résultats et leurs preuves}
Le théorème~\ref{theo:R} donne une première minoration de $\rr$ par une constante positive, les minorants précédents tendent exponentiellement vers $0$ (même ceux ne concernant que le cas non compact, voir \cite{friedland96}). Sa preuve est extrêment simple~: on prend un point de tangence d'une région cuspidale maximale avec elle-même, et on évalue le rayon d'injectivité en ce point. La détermination du rayon maximal de la variété de Gieseking demande un peu plus d'efforts~: nous décomposons la variété en deux polyèdres, et nous utilisons des propriétés de convexité pour minorer le rayon d'injectivité à l'intérieur des polyèdres.\par
  En ce qui concerne l'inégalité $\cosh(\sys/2)\leq c_n\cdot \vol_\triangle$, nous nous inspirons du travail d'Adams et Reid pour construire une géodésique dont nous contrôlons la longueur en fonction du réseau euclidien à l'infini. Nous comparons cette longueur au volume de la variété grâce à des méthodes classiques de minoration de volume (voir partie~\ref{sec:volume}).\par 
 Le résultat principal de l'article est la majoration optimale de la systole en dimension~$3$. Nous établissons l'inégalité en nous ramenant à un probléme d'empilement de disques dans les surfaces euclidiennes. Ce problème est résolu \emph{via} un argument variationnel (partie~\ref{sec:surfaces-plates}). Nous déterminons le cas d'égalité grâce à une caratérisation du n\oe ud de huit en termes de longueur de lacets périphériques due à C.~Adams (\cite{adams02}). Le lemme suivant (voir \textsection~\ref{sec:translation}) s'avère essentiel pour estimer les longueurs des géodésiques
\begin{lemmanonumber}
Soit $\g$ un élément loxodromique de $\PSL(2,\C)$. Sa distance de translation $\ell_\g$ est donnée par $2\cosh(\ell_\g/2)=|\Tr(\g)/2-1|+|\Tr(\g)/2+1|$.
\end{lemmanonumber} 
 Dans les énoncés des résultats, nous préférons faire intervenir le volume simplicial plutôt que le volume. Ceci est un peu artificiel car une $n$-variété hyperbolique de volume fini vérifie $\vol=\nu_n\cdot \vol$ (voir \cite{thurston} \textsection~$6.5$), où $\nu_n$ désigne le volume du simplexe idéal régulier de $\Hyp^n$.

\subsection{Plan de l'article}
Nous effectuons d'abord des rappels sur la structure et le volume des variétés hyperboliques non compactes (partie~\ref{sec:volume}). Nous établissons ensuite la minoration du rayon maximal (partie~\ref{sec:rayon-maximal}), puis les inégalités entre systole et volume simplicial (partie~\ref{sec:systole-asymptotique}). Le reste de l'article est consacré à l'inégalité optimale sur la systole en dimension $3$. Elle nécessite des lemmes spécifiques à la dimension $3$ (partie~\ref{sec:preliminaires}), ainsi qu'une proposition sur les empilements de disques dans les surfaces plates (partie~\ref{sec:surfaces-plates}). Tous ces éléments sont utilisés pour finalement établir le résultat (partie~\ref{sec:ineg-optimale}).

\subsection{Remerciements}
Je tiens à remercier Ruth Kellerhals pour son soutien. Je la remercie aussi pour m'avoir indiqué de nombreuses références et prêté plusieurs documents personnels. Une partie de ce travail a été réalisée lors d'un séjour au Max-Planck-Institut für Mathematik de Bonn, je remercie la Max-Planck-Gesellschaft de son soutien et les membres de l'institut pour leur accueil.

\section{Volume et bouts cuspidaux}\label{sec:volume}

 Dans cette partie, nous rappelons plusieurs résultats classiques portant sur la structure et le volume des variétés hyperboliques non compactes. Les idées relatives au volume sont issues des travaux \cite{meyerhoff} de R.~Meyerhoff, \cite{adams} de C.~Adams, et \cite{kellerhals} de R.~Kellerhals. Ce dernier nous sert de référence principale sur ce sujet.
 
\subsection{Groupes kleiniens}

 Soit $M$ une variété hyperbolique de volume fini et de dimension $n\geq 3$. Il existe une représentation fidèle $\G$ de $\pi_1(M)$ comme réseau de $\mathsf{Isom}(\Hyp^n)$. Selon le théorème de rigidité de Mostow et Prasad, la représentation $\G$ est unique à conjugaison près par une isométrie de $\Hyp^n$, et le quotient $\G\backslash\Hyp^n$ s'identifie isométriquement à $M$.\par
 Lorsque $M$ est non compacte, le groupe $\G$ admet des stabilisateurs paraboliques pour son action sur le bord conforme $\partial\Hyp^n$. Ces stabilisateurs sont en nombre fini modulo conjugaison dans $\G$. Considérons $\G_x$ ($x\in\partial\Hyp^n$) l'un de ces stabilisateurs, il agit de façon conforme et libre sur le complémentaire $\partial\Hyp^n\setminus\{x\}$, et s'identifie ainsi à un réseau sans torsion de $\Isom(\Eu^{n-1})$. Cette représentation de $\G_x$ comme réseau de $\Isom(\Eu^{n-1})$ n'est pas canonique, mais elle est bien définie à conjugaison près par une similitude.\par
 Dans la suite nous supposerons $M$ non compacte, nous désignerons par $\G_x$ un stabilisateur parabolique d'un point $x\in\partial\Hyp^n$. Nous noterons $\L_x=\L(\G_x)$ le sous-groupe abélien libre maximal de $\G_x$. Notez que l'indice de $\L_x$ dans $\G_x$ est borné par une fonction de $n$. 

\subsection{Cuspides}
 Une \emph{région cuspidale} est une partie $C$ de $M$ de la forme $\G_x\backslash B_x$, où $B_x$ est une horoboule centrée en $x\in\partial\Hyp^n$ ne contenant pas deux points identifiés par un élément $\g\in\G\setminus\G_x$. Nous appellerons \emph{cuspide} ou \emph{bout cuspidal} une classe d'équivalence $\mathcal{C}$ de régions cuspidales pour la relation $C\sim C'$ si $C$ contient $C'$ ou $C'$ contient $C$.\par
 L'inclusion, ou de manière équivalente le volume, ordonne totalement une cuspide. La cuspide admet un \emph{représentant maximal} pour cette relation d'ordre consistant en l'union de ses éléments. Nous appellerons \emph{volume} de la cuspide le volume de son représentant maximal.\par

\subsection{Volume des cuspides}
 Une région cuspidale $C$ se relève à $\Hyp^n$ en un empilement d'horoboules. Supposons l'une de ces horoboules de la forme $B_\infty=\{x_n>h\}$ avec $h>0$. Soit $P_\infty$ un domaine fondamental mesurable pour l'action de $\Gamma_\infty$ sur $\Eu^{n-1}=\partial\Hyp^{n}\setminus\{\infty\}$. Alors
\begin{equation}\label{eq:volume1}
\vol(C)= \vol(P_\infty) \int_h^\infty \frac{\mathrm{d}x_n}{x_n^n} =\frac{\covol(\G_\infty)}{(n-1)h^{n-1}}=\frac{\covol(\L_\infty)}{\ic(n-1) h^{n-1}},
\end{equation}
où $\ic$ désigne l'indice de $\L_\infty$ dans $\G_\infty$. Le covolume de $\G_\infty$ est mesuré relativement à son action sur $\Eu^{n-1}$.\par
 Si $C$ est le représentant maximal de la cuspide, alors $B_\infty$ admet des horoboules tangentes. Celles-ci se projettent orthogonalement sur $\Eu^{n-1}$ pour former un empilement de boules de diamètre $h$. Cet empilement est invariant par $\G_\infty$, et passe au quotient en un empilement de boules de diamètre $h$ dans la variété plate $\G_\infty\backslash \Eu^{n-1}$. À un facteur multiplicatif près, la quantité $h^{n-1}/\covol(\G_\infty)$ représente la densité de cet empilement.
 
\subsection{Le facteur $2$ d'Adams}\label{sec:facteur2}
 Par maximalité de $C$, il existe un élément $\g\in\G\setminus\G_\infty$ identifiant deux points de $\partial B_\infty$. Les isométries $\g$ et $\g^{-1}$ envoient $B_\infty$ sur des horoboules qui lui sont tangentes aux points identifiés par $\g$. Pour fixer les idées, posons $\g^{-1}(B_\infty)=B_0$ et $\g(B_\infty)=B_b$ avec $b\in\Eu^{n-1}$.\par
 Les horoboules $B_0$ et $B_b$ n'appartiennent pas à la même orbite sous l'action de $\G_\infty$. Sinon il existerait $\tau\in\G_\infty$ tel que $\tau(b)=0$, et $\tau\g$ fixerait le point de tangence entre $B_0$ et $B_\infty$, ce qui est impossible car $\G$ n'a pas d'élément elliptique.\par
  Il s'ensuit que l'empilement de boules de diamètre $h$ dans la variété plate $\G_\infty\backslash \Eu^{n-1}$ contient au moins deux boules. C.~Adams mis en évidence ce fait important dans \cite{adams}.
   
\subsection{Densité des cuspides}\label{sec:densite}
  Le passage des cuspides à la variété toute entière s'effectue au moyen du théorème de K.~Böröczky (\cite{boroczky} \textsection~6). L'idée d'utiliser ce théorème pour estimer la densité des cuspides dans une $3$-variété hyperbolique revient à R.~Meyerhoff (\cite{meyerhoff}). On trouve cependant dans l'article \cite{kellerhals} de R.~ Kellerhals le premier énoncé clair valable en toute dimension (lemme~$3.2$)~:

\begin{theoremnonumber}[R.~Meyerhoff, R.~Kellerhals]\label{theo:meyerhoff}
Soit $\mathscr{C}$ une union de régions cuspidales disjointes d'une variété hyperbolique $M^n$. La densité de $\mathscr{C}$ dans $M$ satisfait 
$$ \frac{\vol(\mathscr{C})}{\vol(M)}\leq d_n(\infty),$$
où $d_n(\infty)$ désigne la densité simpliciale des horoboules dans $\Hyp^n$ définie ci-dessous.
\end{theoremnonumber}

\subsection{Densité simpliciale}\label{sec:constantes}
La \emph{densité simpliciale des horoboules} dans $\Hyp^n$ est la densité de $n+1$ horoboules mutuellement tangentes de $\Hyp^n$ dans le simplexe idéal régulier engendré par leurs centres. L'application du théorème ci-dessus est conditionnée à la connaissance de la constante $d_n(\infty)$, pour laquelle R.~Kellerhals obtint dans \cite{kellerhals-crelle} la formule suivante~:

\begin{theoremnonumber}[R.~Kellerhals]
La densité simpliciale des horoboules est donnée par~:
\begin{eqnarray*}
d_n(\infty) & = & \frac{n+1}{n-1} \frac{n}{2^{n-1}} \prod_{k=2}^{n-1} \left(\frac{k-1}{k+1} \right)^{\frac{n-k}{2}} \frac{1}{\nu_n},
\end{eqnarray*}
où $\nu_n$ désigne le volume du simplexe idéal régulier de $\Hyp^n$.
\end{theoremnonumber}

En procédant à la simpification suivante~:
\begin{eqnarray*}
 \prod_{k=2}^{n-1} \left(\frac{k-1}{k+1} \right)^{\frac{n-k}{2}} & = & \left( \prod_{l=1}^{n-2} l^{n-l-1} \right)^{\frac{1}{2}} \left( \prod_{l=3}^{n} l^{n-l+1}   \right)^{-\frac{1}{2}} \\
 & = & \left( \prod_{l=3}^{n-2} \frac{1}{l} \right)\ \frac{2^{\frac{n-3}{2}}}{(n-1)\sqrt{n}} \\
 & = & 2^{\frac{n-1}{2}}\ \frac{\sqrt{n}}{n!},
\end{eqnarray*} 
nous obtenons une nouvelle expression de $d_n(\infty)$ reflétant son comportement asymptotique~: 
\begin{eqnarray*}
d_n(\infty) & = & \frac{n+1}{n-1}\ \frac{\sqrt{n}}{(n-1)!}\  \frac{2^{-\frac{n-1}{2}}}{\nu_n}.
\end{eqnarray*}
J.~Milnor donna dans \cite{milnor} p.$207$ une formule explicite exprimant le volume du simplexe idéal régulier de $\Hyp^n$, il en déduisit l'équivalent $\nu_n\simeq e\sqrt{n}/n!$ (voir aussi \cite{haagerup} p.$11$). Nous arrivons ainsi à~:
\begin{eqnarray*}
d_n(\infty) & \simeq & \frac{n}{e\  2^{\frac{n-1}{2}}}.
\end{eqnarray*}
En particulier $d_n(\infty) \simeq \sqrt{2} d_n$ où $d_n$ désigne la \emph{densité simpliciale} dans $\Eu^n$, c'est-à-dire la densité de $n+1$ boules euclidiennes de même rayon mutuellement tangentes dans le simplexe régulier engendré par leurs centres. L'asymptotique de $d_n$ a été calculée par H.~E.~Daniels dans \cite{rogers}, on la retrouve rapidement en utilisant les formules de R.~Kellerhals (\cite{kellerhals-crelle}) et  T.~H.~Marshall (\cite{marshall}).

\subsection{Minoration du volume}\label{sec:minoration-volume}
Le théorème de Meyerhoff et Kellerhals combiné avec \eqref{eq:volume1} donne directement
\begin{equation}\label{eq:minoration-volume}
\vol(M)\geq \frac{\covol(\G_\infty)}{(n-1)d_{n}(\infty)h^{n-1}}= \frac{\covol(\L_\infty)}{\ic (n-1)d_{n}(\infty)h^{n-1}}.
\end{equation}
C'est cette minoration que nous utiliserons dans la partie \ref{sec:systole-asymptotique}. Poursuivons néanmoins avec quelques commentaires sur le volume minimal des variétés hyperboliques non compactes.\par
 Suivant la remarque d'Adams, la variété plate $\G_\infty\backslash\Eu^{n-1}$ contient au moins deux boules disjointes de diamètre $h$. Or, C. A. Rogers a montré dans \cite{rogers} que la densité d'un empilement de boules de même rayon dans une variété plate est inférieure ou égale à la densité simpliciale $d_n$. On en déduit immédiatement (théorème $3.5$ de \cite{kellerhals})
\begin{equation*}
2\ \frac{(h/2)^{n-1}\omega_{n-1}}{\covol(\G_\infty)}\leq d_n\quad\mathrm{et}\quad\vol(M) \geq \frac{\omega_{n-1}}{2^{n-2}(n-1) d_{n-1}d_n(\infty)},
\end{equation*}
où $\omega_{n-1}$ désigne le volume de la boule unité de $\Eu^{n-1}$. Remarquez le facteur $2$ dû à C.~Adams.\par
 En dimension $3$ on évalue facilement ces différentes quantités, on trouve ainsi
\begin{equation*}
\sqrt{3}h^2\leq \covol(\G_\infty)\quad\mathrm{et}\quad \vol(M^3)\geq \nu_3.
\end{equation*}
C.~Adams obtint ces inégalités optimales dans \cite{adams}, où il montra de plus que la variété de Gieseking (voir \textsection~\ref{sec:Gieseking}) est l'unique $3$-variété hyperbolique non compacte de volume minimal.\par 
R.~Kellerhals calcula dans \cite{kellerhals-crelle} des formules exprimant $d_{n-1}$ et $d_{n}(\infty)$ en fonction de certains volumes, elle en déduisit dans \cite{kellerhals} la minoration suivante valable pour toute variété hyperbolique non compacte de dimension $n\geq 3$~:
\begin{equation*}
\vol(M^n)> \frac{2^n}{n(n+1)} \nu_n.
\end{equation*}
Ceci montre une croissance asymptotique du volume simplicial minimal.

\section{Minoration du rayon maximal}\label{sec:rayon-maximal}

\subsection{Une borne optimale}

\begin{theorem}
Soit $M$ une variété hyperbolique non compacte. Soit $P$ un point de tangence du représentant maximal d'une cuspide de $M$ avec lui même, alors
$$
\rr(M)\geq \re_{inj}(P) \geq  \mathrm{arccosh}(\sqrt{5}/2) .
$$
La variété de Gieseking réalise l'égalité entre ces différentes quantités.
\end{theorem}

\begin{remark}
Nous ne savons pas si la variété de Gieseking est la seule variété réalisant l'égalité.
\end{remark}

\begin{proof}
Soit $C$ un représentant maximal d'une cuspide $\mathcal{C}$ de $M$. Nous nous plaçons dans la situation habituelle où $B_0$ et $B_\infty=\{x\in\Hyp^n~;~x_n> 1\}$ sont deux horoboules tangentes au-dessus de $C$. Appelons $\tilde{P}$ le point de tangence entre ces horoboules, et $P$ son image dans $M$. Nous avons trivialement $\max_M \re_{inj} \geq \re_{inj}(P)$.\par
 Pour un sous-ensemble $A\subset\Hyp^n$ et un point $x\in\partial\Hyp^n$, notons $\mathsf{Cone}_x(A)$ l'union des géodésiques d'extrémité $x$ passant par un point de $A$. Sur l'horosphère $\partial B_0$ (resp. $\partial B_\infty$) munie de sa métrique induite, on considère le disque $D_0$ (resp. $D_\infty$) centré en $\tilde{P}$ et de rayon $1/2$. Alors
 
\begin{lemma}
Le rayon d'injectivité $\re_{inj}(P)$ est supérieur ou égal au rayon de la plus grande boule contenue dans l'intersection des cônes $\mathsf{Cone}_0(D_0)\cap \mathsf{Cone}_\infty(D_\infty)$. Cette boule est centrée en $\tilde{P}$ et de rayon $\mathrm{arccosh}(\sqrt{5}/2)$.
\end{lemma} 

\begin{proof}[Démonstration du lemme]
Le double cône $\mathsf{Cone}_0(D_0)\cap \mathsf{Cone}_\infty(D_\infty)$ (figure \ref{fig:cones}) admet une symétrie de révolution par rapport à la droite $(0\infty)$, ainsi qu'une symétrie par rapport à l'hypersurface géodésique tangente aux horoboules $B_0$ et $B_\infty$. En dimension $2$, cet ensemble consiste en le quadrilatère convexe $(0 e^{i\pi/3} \infty e^{2i\pi/3})$, que l'on reconnaît comme l'union de deux domaines fondamentaux pour l'action habituelle de $\PSL(2,\Z)$ sur $\Hyp^2$. Clairement, la plus grosse boule contenue dans l'intersection des cônes est centrée en $\tilde{P}$ et de rayon $\mathrm{arccosh}(\sqrt{5}/2)$.\par
 Nous allons montrer que l'intersection des cônes est contenue dans la cellule de Dirichlet centrée en $\tilde{P}$, ceci suffit pour conclure. Nous ne traiterons que le cas de la dimension $2$.  Soit $z$ un point de $\mathsf{Cone}_0(D_0)\cap \mathsf{Cone}_\infty(D_\infty)$, en raison de la symétrie nous supposerons $|z|\geq 1$. Par construction, $\tilde{P}=i$ minimise la distance à $z$ parmi les relevés de $P$ appartenant à l'horosphère $\partial B_\infty=\{\mathsf{Im}=1\}$. Les autres relevés de $P$ appartiennent à la bande $\{\mathsf{Im}\leq 1/2\}$, car ce sont des points de tangence entre des horoboules de l'empilement. On vérifie facilement que $z$ est au moins aussi proche de $\tilde{P}$ que de n'importe quel point de $\{\mathsf{Im}\leq 1/2\}$, donc de n'importe quel autre relevé de $P$. En conclusion $z$ appartient à la cellule de Dirichlet centrée en $\tilde{P}$.
\end{proof}

Il reste à montrer qu'il y a égalité lorsque $M$ est la variété de Gieseking, ceci fait l'objet du paragraphe suivant.
\end{proof}

\begin{figure}[h]
\psfrag{BI}{$B_\infty$}\psfrag{B0}{$B_0$}\psfrag{P}{$\tilde{P}$}\psfrag{C}{$\mathsf{Cone}_0(D_0)\cap \mathsf{Cone}_\infty(D_\infty)$}
\centering
 \includegraphics[totalheight=6cm,keepaspectratio=true]{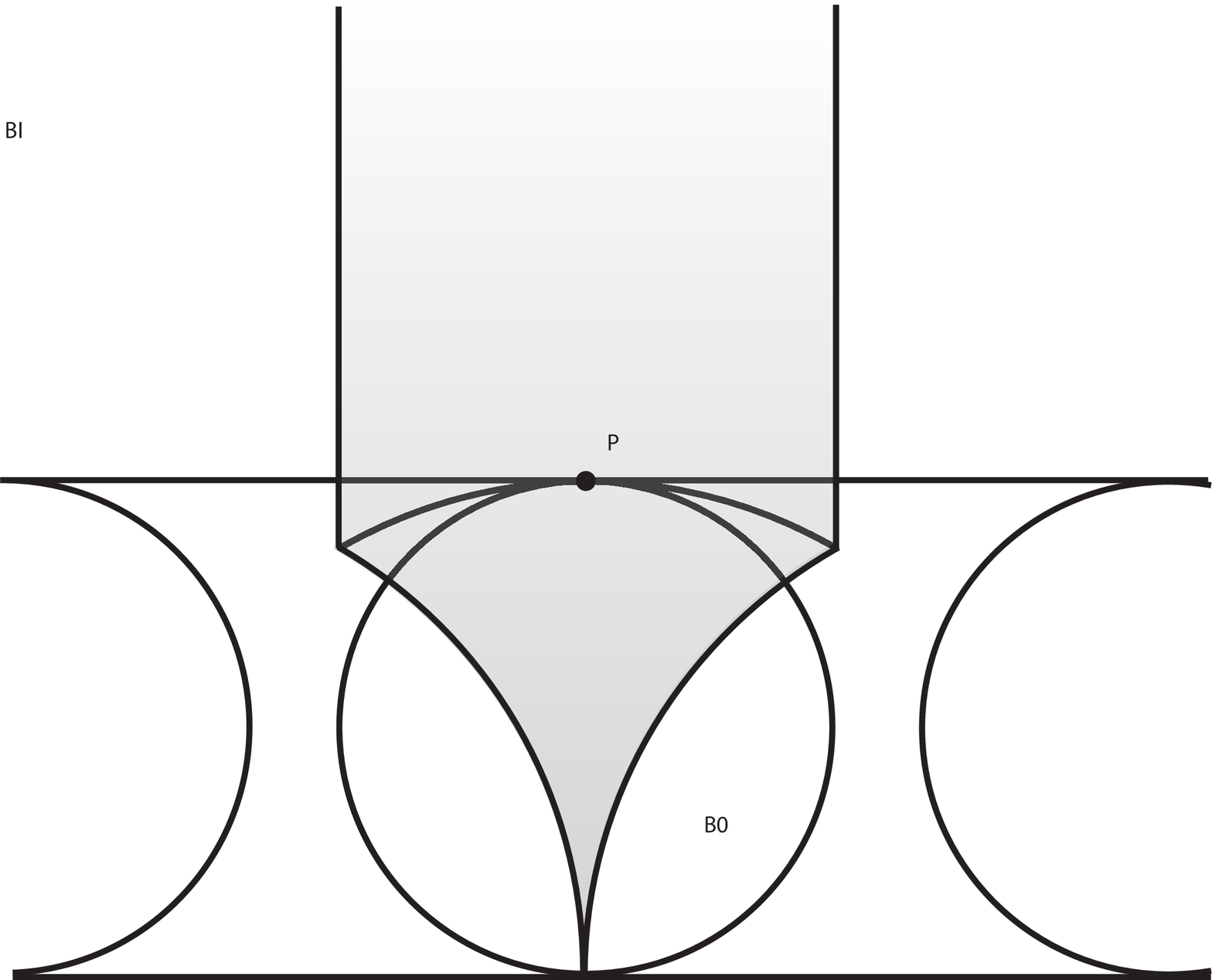}
\caption{Le double cône $\mathsf{Cone}_0(D_0)\cap \mathsf{Cone}_\infty(D_\infty)$}\label{fig:cones}
\end{figure}

\subsection{Le rayon d'injectivité de la variété de Gieseking}\label{sec:Gieseking}

Soit $\mathcal{S}$ un simplexe idéal régulier de $\Hyp^3$. Les identifications des faces de $\mathcal{S}$ représentées en figure~\ref{fig:Gieseking}  réunissent toutes les arêtes dans un même cycle. Ainsi, selon un théorème de Poincaré, les isométries réalisant ces identifications engendrent un sous-groupe discret de $\Isom(\Hyp^3)$, construit explicitement au \textsection~\ref{sec:systole-gieseking}. Les identifications ne fixant aucun point, ce sous-groupe agit librement et l'espace quotient associé est une variété hyperbolique appelée \emph{variété de Gieseking} et notée N$1_1$ dans le catalogue SnapPea (\cite{weeks}). Elle réalise le minimum du volume parmi les $3$-variété hyperboliques non compactes (voir \textsection~\ref{sec:minoration-volume}).\par
 Chaque face du simplexe admet un cercle inscrit, les points de tangence de ces cercles avec les arêtes de $\mathcal{S}$ se projettent sur un même point $P$ dans la variété. Ce point est le seul point de tangence du représentant maximal de la cuspide de $N1_1$ avec lui-même. On évalue facilement la distance minimale entre deux relevés de $P$, on trouve $\re_{inj}(P)=\mathrm{arccosh}(\sqrt{5}/2)$, à comparer avec le rayon de la boule inscrite dans $\mathcal{S}$ qui vaut $\mathrm{arccosh}(3/2\sqrt{2})$.\par
 Soient $v_1,\ldots,v_4$ les centres des cercles inscrits dans les faces de $\mathcal{S}$, et $v$ le centre de la boule inscrite dans $\mathcal{S}$. Les $v_i$ sont les points de tangence de la boule inscrite. À chaque paire d'indices distincts $\{i,j\}$ on associe l'enveloppe convexe des points $v$, $v_i$, $v_j$ et de l'arête adjacente aux faces contenant $v_i$ et $v_j$ (figure~\ref{fig:decomposition1}). De cette façon on décompose $\mathcal{S}$ en $6$ tétraèdres isométriques. Les translatés de ces tétraèdres adjacents à une même arête de $\mathcal{S}$ forme un nouveau domaine fondamental $\mathcal{T}$ à $12$ faces (figure~\ref{fig:dirichlet} et \ref{fig:decomposition2}). Chaque face de $\mathcal{T}$ est supportée par un plan totalement géodésique équidistant de  deux relevés de $P$, on en conclut que ce domaine fondamental consiste en la cellule de Dirichlet centrée en un relevé de $P$.\par
 On peut visualiser les domaines fondamentaux $\mathcal{S}$ et $\mathcal{T}$ grâce au logiciel SnapPea. L'idée de la minoration du rayon maximal vient en partie des observation faites avec SnapPea.  
 
\begin{figure}[p]
\centering
\psfrag{v}{$v$}\psfrag{v1}{$v_1$}\psfrag{P}{$\tilde{P}$}
\begin{minipage}[b]{.48\linewidth}
 \centering\epsfig{figure=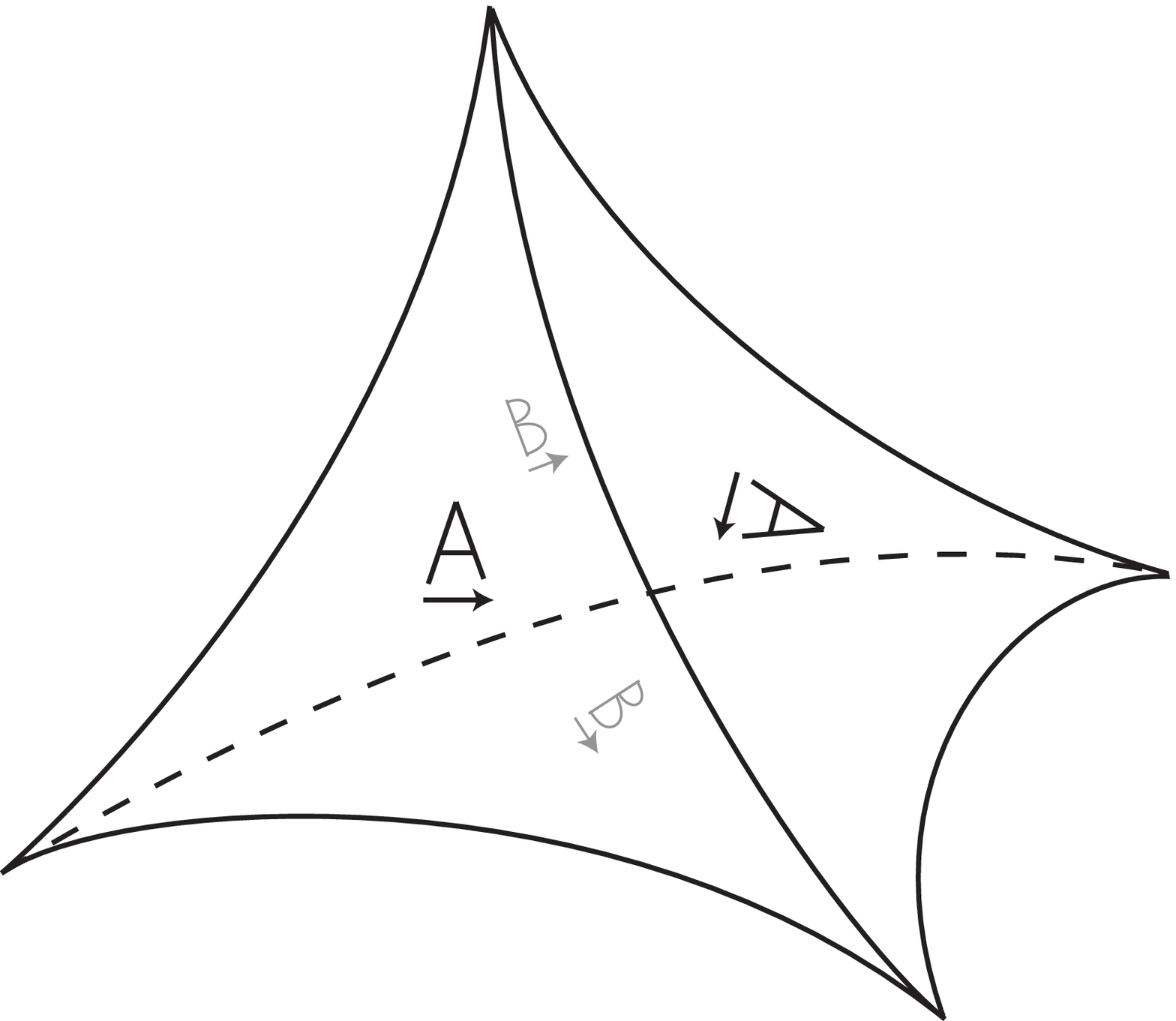,totalheight=6cm}
 \caption{$\mathcal{S}$}\label{fig:Gieseking}
\end{minipage}
\begin{minipage}[b]{.48\linewidth}
 \centering\epsfig{figure=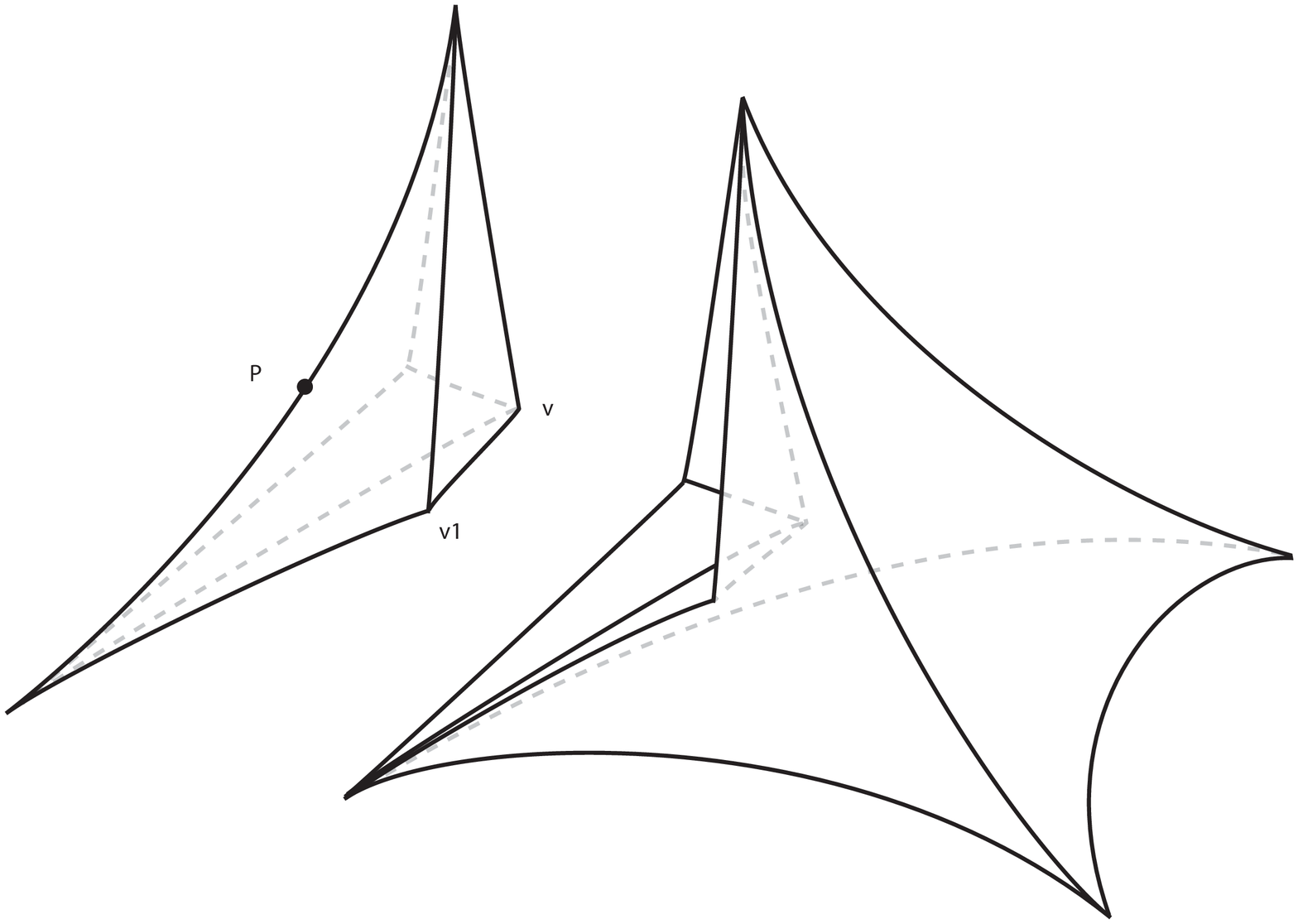,totalheight=6cm}
 \caption{Découpage de $\mathcal{S}$}\label{fig:decomposition1}
\end{minipage}
\vspace{1cm}
\begin{minipage}[b]{.48\linewidth}
\centering\epsfig{figure=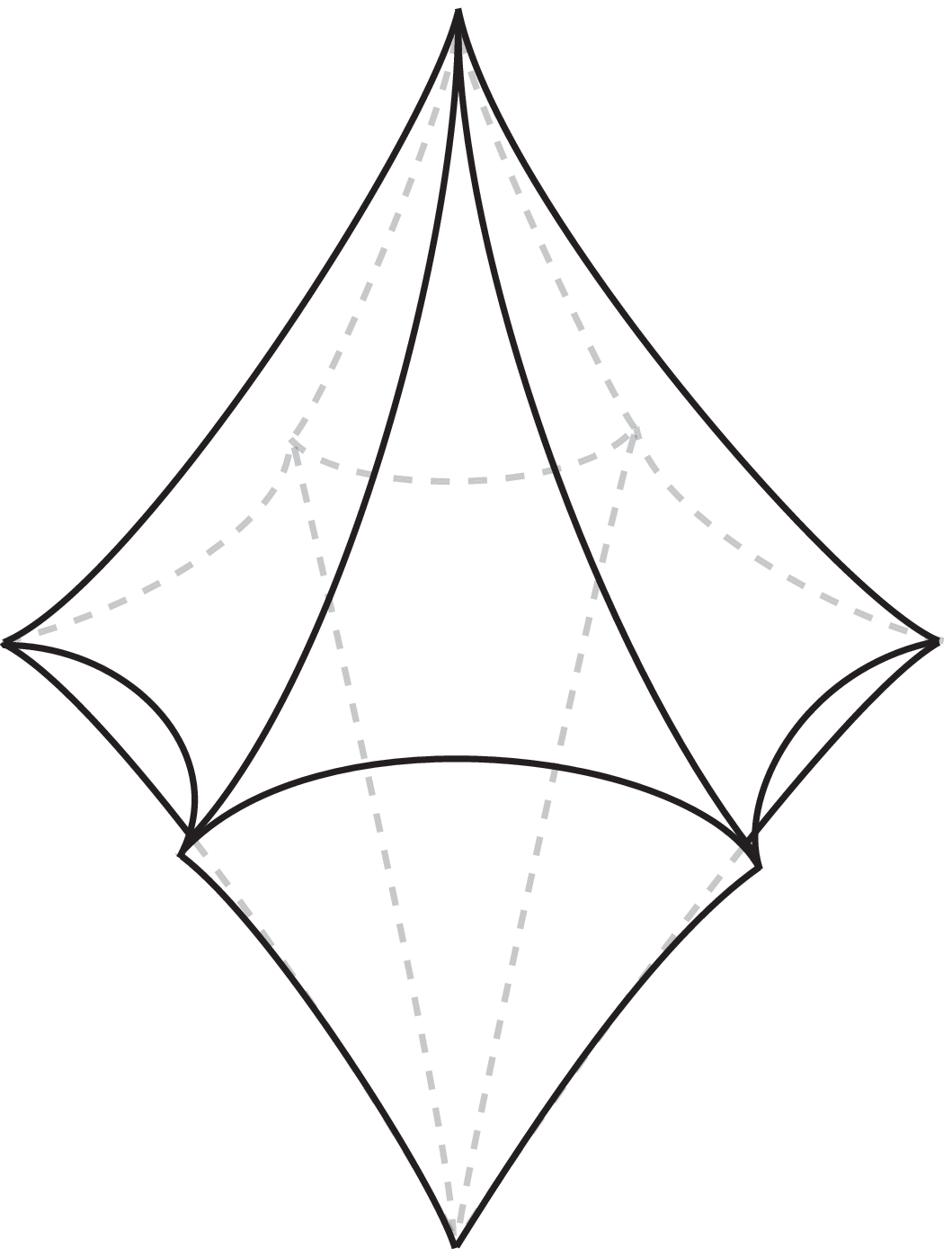,totalheight=6cm}
\caption{$\mathcal{T}$}\label{fig:dirichlet}
\end{minipage} \hfill
\begin{minipage}[b]{.48\linewidth}
\centering\epsfig{figure=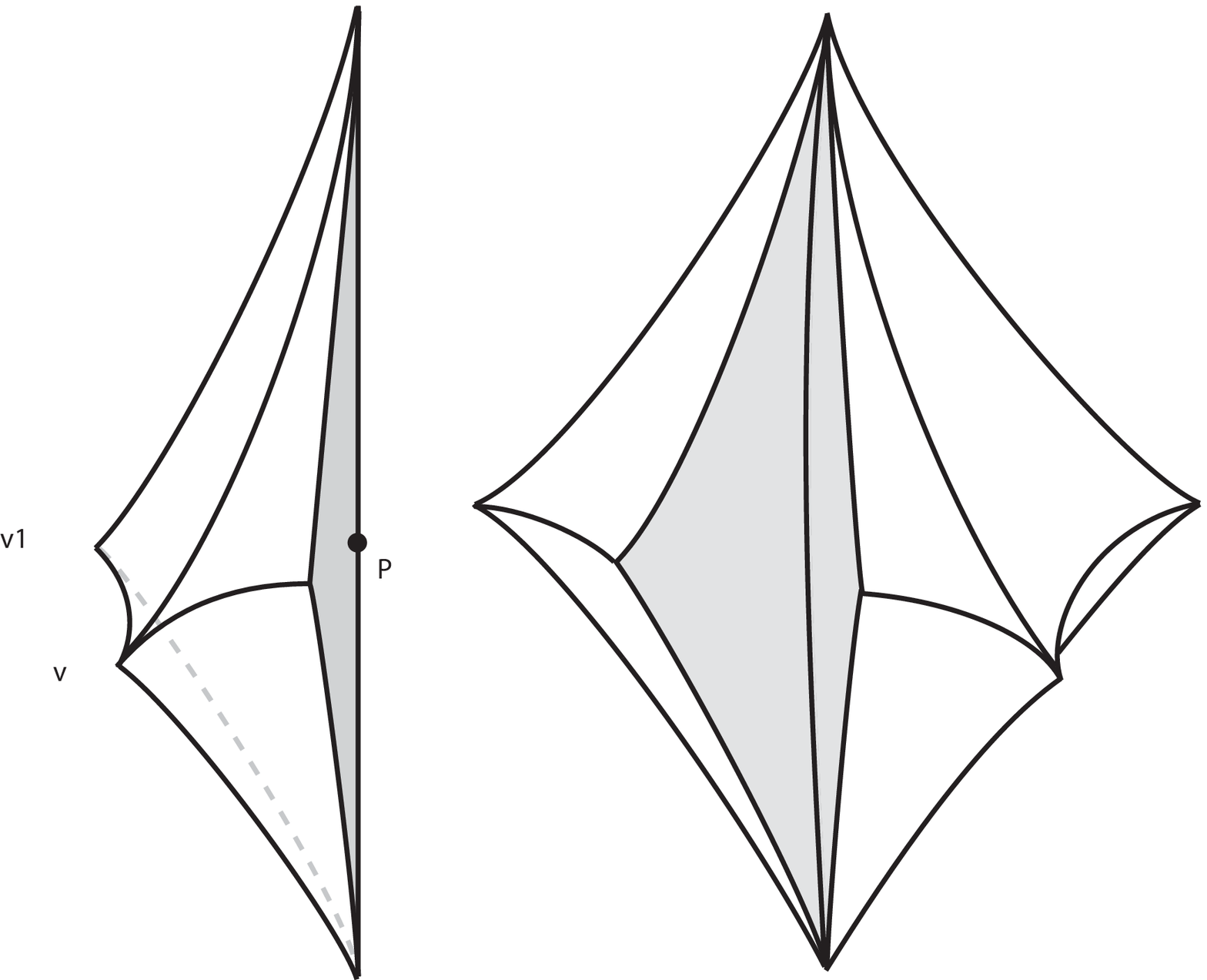,totalheight=6cm}
\caption{Découpage de $\mathcal{T}$}\label{fig:decomposition2}
\end{minipage} \hfill
\begin{minipage}[b]{.48\linewidth}
 \centering\epsfig{figure=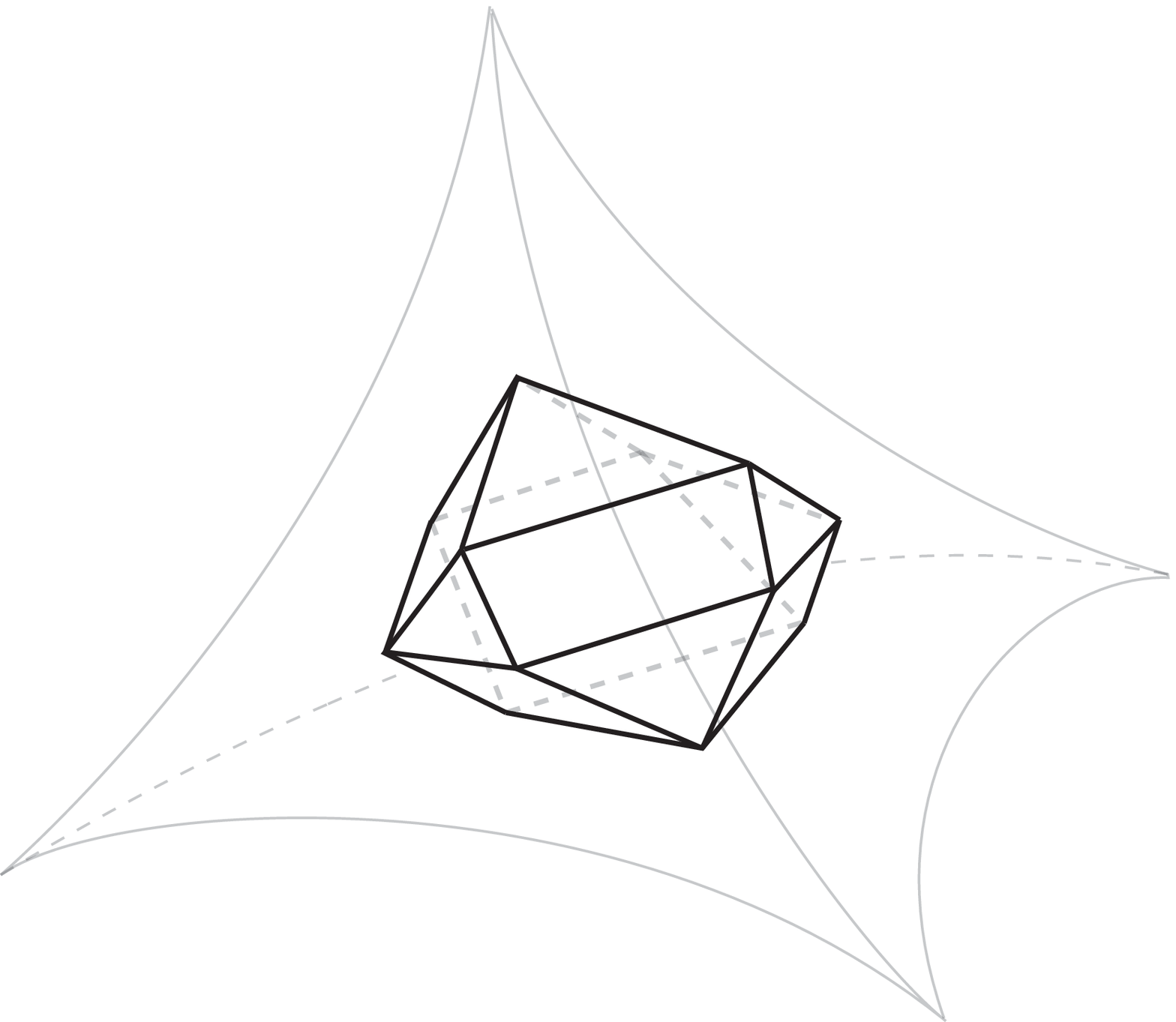,totalheight=6cm}
 \caption{$S$}\label{fig:env-tetraedre}
\end{minipage}
\begin{minipage}[b]{.48\linewidth}
\centering\epsfig{figure=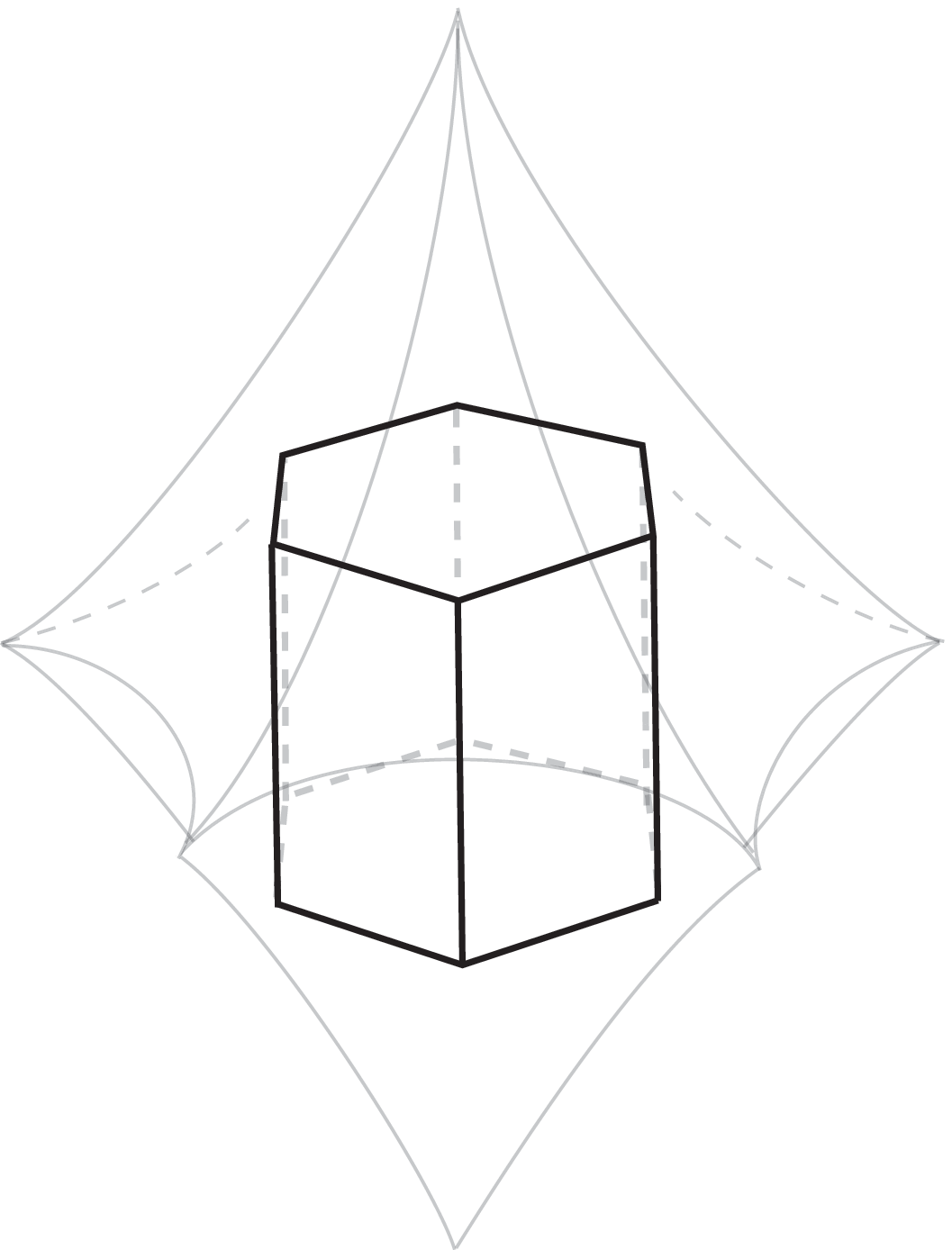,totalheight=6cm}
\caption{$T$}\label{fig:env-dirichlet}
\end{minipage} \hfill
\end{figure}

\begin{lemma}
Le maximum du rayon d'injectivité de la variété de Gieseking est atteint au point $P$ uniquement.
\end{lemma}

\begin{proof}
Soit $C$ le représentant maximal de la cuspide, le point $P$ réalise le maximum du rayon d'injectivité sur $C$. En effet, ce point se situe sur le bord de $C$, et l'intersection de la boule $\mathsf{B}(P,\re_{inj}(P))$ avec une certaine fibre $C_0$ de  $C$ consiste en un empilement optimal de deux disques de même rayon dans une bouteille de Klein plate.\par
 La boule $\mathsf{B}(P,\re_{inj}(P))$ a $6$ points de tangence avec elle-même, appartenant tous à $C_0$. Chacun d'eux admet deux relevés dans $\mathcal{S}$ et deux relevés dans $\mathcal{T}$. L'enveloppe convexe des relevés dans $\mathcal{S}$ est un polyèdre $S$ à $14$ faces ($8$ triangles et $6$ quadrilatères) représenté en figure~\ref{fig:env-tetraedre}, tandis que l'enveloppe convexe des relevés dans $\mathcal{T}$ est un polyèdre $T$ à $8$ faces ($6$ quadrilatères et $2$ hexagones) représenté en figure~\ref{fig:env-dirichlet}. Vus dans $M$, ces polyèdres sont d'intérieurs disjoints, mais collés l'un à l'autre suivant les $6$ quadrilatères. L'union des deux polyèdres recouvre le complémentaire de $C$ dans $M$. Il suffit donc d'étudier $\re_{inj}$ dans $S$ et dans $T$ pour conclure. \par 
  Les sommets de $S$ sont à une distance $\mathrm{arccosh}(\sqrt{6}/\sqrt{5})$ du centre de $S$, et les sommets de $T$ à une distance $\mathrm{arccosh}(\sqrt{5}/2)$ du centre de $T$. En tenant compte des identifications, chaque polyèdre a exactement $18$ arêtes, toutes de longueur $\mathrm{arccosh}(11/10)$. En effet, les sommets d'un triangle quelconque de $S$ sont exactement les milieux des côtés d'un triangle équilatéral dont les sommets sont des relevés de $P$.\par
 Soit $F$ une face de $S$ contenue dans une face de $\mathcal{S}$. Les points de $F$ sont à une distance au plus $\mathrm{arccosh}(\sqrt{6}/\sqrt{5})$ de $v$ et de l'un de ses translatés $\g v$. Par convexité, un point $x$ dans le cône de sommet $v$ et de base $F$ satisfait $d(x,v)+d(x,\g v)\leq 2 ~\mathrm{arccosh}(\sqrt{6}/\sqrt{5})$, d'où $\re_{inj}(x)\leq \mathrm{arccosh}(\sqrt{6}/\sqrt{5})<\mathrm{arccosh}(\sqrt{5}/2)$.\par
 Soit $F$ une face de $S$ ou $T$ qui n'est pas contenue dans une face de $\mathcal{S}$. On vérifie sans difficulté qu'un point de $F$ se trouve à une distance inférieure à $\mathrm{arccosh}(\sqrt{5}/2)$ de deux sommets de $F$ appartenant à une même horosphère au-dessus de $C_0$. Cette propriété s'étend par convexité  à tout point à l'intérieur du cône de base $F$ et de sommet le centre du polytope. Ainsi, une boule de rayon $\mathrm{arccosh}(\sqrt{5}/2)$ centrée en un point de ce cône contient deux sommets appartenant à une même horosphère au-dessus de $C_0$. La trace de la boule sur l'horosphère consiste en des disques euclidiens dont le diamètre est au moins la distance entre les deux sommets, mais ce diamètre dépasse le diamètre du plus grand disque plongé dans $C_0$. Le rayon d'injectivité en tout point du cône est donc inférieur à $\mathrm{arccosh}(\sqrt{5}/2)$.
\end{proof}

\section{Inégalité systolique asymptotique}\label{sec:systole-asymptotique}

\subsection{Sur le produit de deux transformations paraboliques}

\begin{proposition}\label{pro:produit}
Soient $\alpha$ et $\beta$ deux éléments purement paraboliques de $\Mob(\Sph^n)$, de points fixes distincts. Si le produit $\a\beta$ est parabolique, alors $\a$ et $\b$ stabilisent une même droite conforme de $\Sph^n$. Sinon $\a$ et $\b$ stabilisent un même plan conforme de $\Sph^n$, sur lequel le produit $\a\b$ admet deux points fixes distincts.  
\end{proposition}

\begin{proof}
Soient $\a$ et $\b$ deux transformations purement paraboliques, de points fixes respectifs $p_\a$ et $p_\b$. Nous noterons $\mathcal{D}_\a$ (resp. $\mathcal{D}_\b$) l'ensemble des droites conformes de $\Sph^n$ stables par $\a$ (resp. $\b$). Considérons $D_\a$ (resp. $D_\beta$) l'unique droite de $\mathcal{D}_\a$ (resp. de $\mathcal{D}_\b$) passant par $p_\b$ (resp. $p_\a$). Les transformations $\a$ et $\b$ stabilisent une même droite conforme de $\Sph^n$ si et seulement si $D_\a$ et $D_\b$ coïncident.\par
 Supposons $D_\a\neq D_\b$, et appelons $P_{\a,\b}$ le plan conforme engendré par $D_\a$ et $D_\beta$. Ce plan est stable par $\a$ et $\beta$. La droite $D_\beta$ intersecte transversalement chaque droite de $\mathcal{D_\a}$ contenue dans $P_{\a,\b}$ en exactement deux points, dont l'un est $p_\a$. Ainsi, $D_\b$ partage $P_{\a,\b}$ en deux cellules $C_\b^+$ et $C_\b^-$, vérifiant
 $$\a(\bar{C}_\b^+)\subset C_\b^+\cup\{p_\a\}\quad \mathrm{et}\quad \a^{-1}(\bar{C}_\b^-)\subset C_\b^-\cup\{p_\a\}.$$
  De même, la droite $D_\a$ partage $P_{\a,\b}$ en deux cellules $C_\a^+$ et $C_\a^-$ vérifiant
$$\b(\bar{C}_\a^+)\subset C_\a^+\cup\{p_\b\}\quad \mathrm{et}\quad \b^{-1}(C_\a^-)\subset C_\b^-\cup\{p_\b\}.$$
Les intersections $C^+=C^+_\a\cap C_\b^+$ et $C^-=C^-_\a\cap C_\b^-$ sont des cellules disjointes vérifiant
$$\a\b(\bar{C^+})\subset C^+\quad \mathrm{et}\quad (\a\b)^{-1}(\bar{C^-})\subset C^-.$$
Par le théorème de Brouwer, il vient que $\a\b$ admet deux points fixes sur $P_{\a,\b}$, et n'est donc pas une transformation parabolique.Nous avons représenté ce qui se passe en figure \ref{fig:sphere}. On y voit les familles de droites $\mathcal{D}_\a$ et  $\mathcal{D}_\b$, ainsi que la cellule $C^+$ en gris.
\end{proof}

\begin{remark}
Si $\a\b$ est parabolique, alors $\langle \a,\b\rangle$ préserve un plan totalement géodésique. L'action de $\langle \a,\b\rangle$ sur ce plan hyperbolique est conjuguée à l'action usuelle du sous-groupe de congruence $\G(2)$ sur $\Hyp^2$. En particulier le quotient est un pantalon à $3$ pointes.
\end{remark}

\begin{figure}[ht]
\centering
\includegraphics[totalheight=6cm]{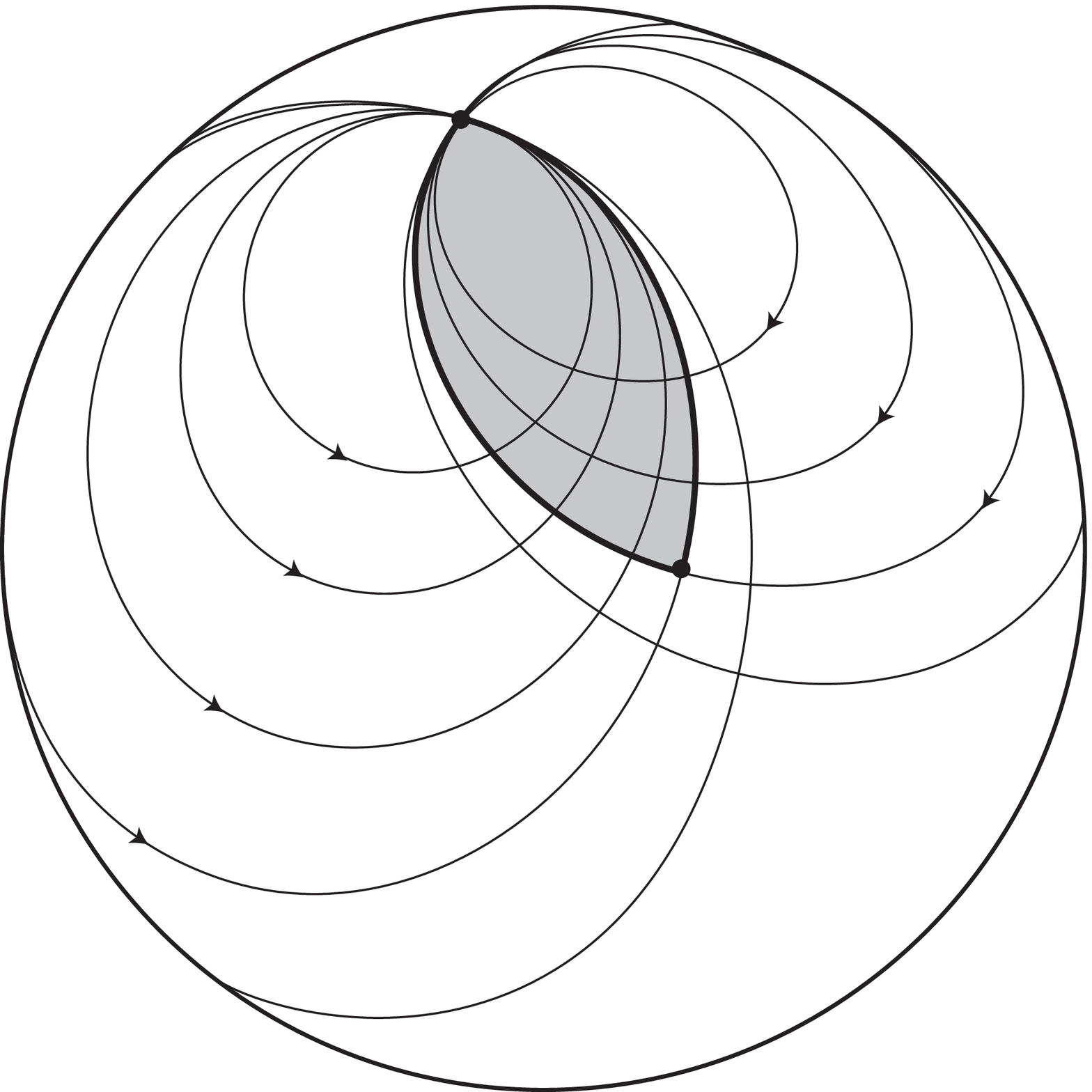}
\caption{$P_{\a,\b}$}\label{fig:sphere}
\end{figure}

\subsection{Inégalité systolique}

Nous établissons maintenant l'inégalité reliant systole et volume simplicial. Cette inégalité fait intervenir la constante d'Hermite $\g_n$, c'est-à-dire la meilleure constante satisfaisant 
$$\min_{\l \in \L\setminus\{0\} } \| \l \|^2\leq \g_n\ \mathrm{pour\ tout\ r\acute eseau\ } \L\subset\Eu^{n}\ \mathrm{de\ covolume\ } 1.$$
  Depuis les travaux de H.~Minkowski et E.~Hlawka, nous connaissons de manière précise le comportement asymptotique de $\g_n$. Selon J.~H.~Conway et N.~J.~A.~Sloane (\cite{conway}), le meilleur encadrement asymptotique actuellement disponible est~: 
 $$\frac{1}{2\pi e} \lesssim \frac{\g_n}{n}\lesssim \frac{1.744}{2\pi e}.$$
Nous renvoyons à leur livre et au livre \cite{martinet} de J.~Martinet pour plus d'informations.

\begin{theorem}
 Soit $M$ une variété hyperbolique non compacte de dimension $n\geq 3$, et soit $\cusp$ une cuspide de $M$. Alors
\begin{eqnarray*}
\frac{\cosh(\sys/2)}{\vol_\triangle}(M) & \leq &\frac{3}{2} \ \frac{\sqrt{n} (n+1)}{(n-1)!}\ \left(\frac{\g_{n-1}}{\sqrt{2}}\right)^{n-1}\ \ic.
\end{eqnarray*}
En injectant le majorant asymptotique ci-dessus, on trouve 
\begin{eqnarray*}
\frac{\cosh(\sys/2)}{\vol_\triangle}(M) & \lesssim & \frac{\ic}{5^{n-1}}.
\end{eqnarray*}
\end{theorem}

\begin{remark}
 Pour $n$ suffisamment grand nous avons $\ic\leq \mathsf{i}_{n-1}\leq 2^{n-1}(n-1)! $,  cette borne est optimale dans le cas des groupes cristallographiques (S.~Friedland \cite{friedland96} \textsection~$5$).
\end{remark}

\begin{proof}
Soit $C$ le représentant maximal de $\cusp$. La région cuspidale $C$ se relève en un empilement d'horoboules, dont deux au moins sont tangentes. Nous nous plaçons dans la situation où deux horoboules tangentes sont centrées en $0$ et $\infty$, nous notons $B_0$ et $B_\infty=\{x_n>h\}$ ces horoboules. Nous supposons de plus la normalisation $\covol(\L_\infty)=1$.\par
 Désignons par $\m_i$ le $i$-ième minima successif de $\L_\infty$~:
$$ \m_i=\inf\{\sup(\|\l_1\|^2,\ldots,\|\l_i\|^2)~;~(\l_1,\ldots,\l_i)\ \mathrm{est\ une\ famille\ libre\ de}\ \L_\infty \}.$$

\begin{lemma}
On a l'inégalité suivante~:
\begin{eqnarray*}
\cosh(\sys(M)/2) & \leq & 1+\frac{\m_1\m_2}{2h^2}.
\end{eqnarray*}
\end{lemma}

\begin{proof}
Soient $\l_1$ et $\l_2$ deux éléments non colinéaires de $\L_\infty$ réalisant $\m_1$ et $\m_2$. Soit $\g$ dans $\G$ tel que $\g\cdot 0=\infty$.
 Les conjugués $\a_1=\g^{-1}\l_1 \g$ et $\a_2=\g^{-1} \l_2\g$ sont des transformations purement paraboliques fixant $0$. En prenant $\beta=\l_1$ et en appliquant la proposition précédente, il vient qu'une des transformations $\a_1\beta$, $\a_2\beta$ est loxodromique. En effet, les transformations $\a_1$ et $\a_2$ ne peuvent fixer la même droite conforme passant par $0$ et $\infty$ puisque $\l_1$ et $\l_2$ ne sont pas colinéaires.\par
 Nous souhaitons estimer la distance de translation de la transformation loxodromique $\a_i\b$. Les transformations $\a_i$ et $\b$ préservent un plan conforme, sur lequel le produit $\a_i\b$ admet deux points fixes (proposition~\ref{pro:produit}). Il suffit donc de travailler dans le sous-espace totalement géodésique bordé par ce plan, car il contient l'axe de la transformation $\a_i\b$. Il revient au même de se placer dans $\Hyp^3$, où nous pouvons représenter $\a_i$ et $\b$ par les matrices
$$\a_i=\left(\begin{array}{cc} 1 & 0 \\ m_i & 1 \end{array}\right)\quad\mathrm{et}\quad \b=\left(\begin{array}{cc} 1 & m \\ 0 & 1 \end{array}\right),$$
où $m$ et $m_i$ sont des nombres complexes de modules $|m|=\m_1$ et $|m_i|=\m_i/h^2$.  On trouve
$$\a_i\beta=\left(\begin{array}{cc} 1 & m \\ m_i & mm_i+1 \end{array}\right),$$
ce qui donne le résultat voulu \emph{via} le lemme~\ref{lem:ellipse}.
\end{proof}

Le lemme ci-dessus associé à la minoration \eqref{eq:minoration-volume} du volume produit l'inégalité
\begin{eqnarray*}
\frac{\cosh(\sys/2)}{\vol}(M) & \leq & \ic (n-1) d_n(\infty) h^{n-1} \left(1+\frac{\m_1 \m_2}{2h^2}\right).
\end{eqnarray*}
Nous avons d'une part $\m_1 \m_2 \leq  \g^2_{n-1}$ par le théorème des minima successif de Minkowski (voir \cite{martinet} p.50), et d'autre part $h\leq \m_1\leq \g_{n-1}$. Nous en déduisons
\begin{eqnarray*}
h^{n-1} \left(1+\frac{\m_1 \m_2}{2h^2}\right) & \leq & \frac{3}{2}\ \g^{n-1}_{n-1}.
\end{eqnarray*}
En remplaçant $d_n(\infty)$ par l'expression du \textsection~\ref{sec:constantes} nous trouvons le majorant souhaité.
\end{proof}

\section{Préliminaires en dimension $3$}\label{sec:preliminaires}

Avant d'attaquer la démonstration de l'inégalité optimale (théorème~\ref{theo:systole}), nous devons nous armer de lemmes spécifiques à la dimension $3$. Après quelques rappels sur le groupe des isométries de $\Hyp^3$, nous exprimons la distance de translation d'un élément loxodromique en fonction de sa trace (\textsection~\ref{sec:translation}). Nous décrivons ensuite le cadre général de notre étude (\textsection~\ref{sec:cadre-general}). Nous étudions en détails les éléments $\g$ d'un réseau $\G$ qui envoient une horoboule $B_0$ sur une horoboule $B_\infty$ (\textsection\textsection~\ref{sec:majo-loxodromique}, \ref{sec:rigidite} et \ref{sec:para-negatif}). Grâce à ces éléments nous pourrons contrôler la systole d'une variété hyperboliques non compacte (partie~\ref{sec:ineg-optimale}). Les notations et les objets introduits dans cette partie seront repris en partie~\ref{sec:ineg-optimale}.\par
 Par abus de langage nous parlerons de la trace d'un élément de $\PSL(2,\C)$ au lieu de parler de la trace d'un de ses relevés à $\SL(2,\C)$. Nous confondrons la sphère $\Sph^2$ et la droite projective complexe $\mathsf{P}^1(\C)$.

\subsection{Le groupe des isométries hyperboliques}
Le groupe  $\mathsf{Mob}(\Sph^2)$ se décompose en le produit semi-direct
$\PSL(2,\C)\rtimes \Z/2\Z$. Ici on identifie $\PSL(2,\C)$ à $\mathsf{Mob}^+(\Sph^2)$ \emph{via} son action par homographie sur $\mathsf{P}^1(\C)$, et $\Z/2\Z$ au sous-groupe engendré par la conjugaison $z\mapsto\bar{z}$. L'automorphisme de $\PSL(2,\C)$ induit par $-1\in\Z/2\Z$ consiste en la conjugaison $A\mapsto\bar{A}$.\par
 Une transformation conforme $\g\in\mathsf{Mob}(\Sph^2)$ s'écrit donc de manière unique
 $$\g:z\mapsto \frac{az+b}{cz+d}\quad \mathrm{ou}\quad \g:z\mapsto \frac{a\bar{z}+b}{c\bar{z}+d}\quad\mathrm{avec}\ a,b,c,d\in\C \ \mathrm{tels\ que}\ ad-bc=1.$$
Cette action par homographies de $\PSL(2,\C)\rtimes \Z/2\Z$ se prolonge à $\Hyp^3\cup\partial\Hyp^3$ tout entier lorsqu'on regarde $\Hyp^3$ comme le sous-ensemble $\{x_1+x_2i+x_3j~; x_1,x_2\in\R\ \mathrm{et}\ x_3> 0\}$ de l'algèbre des quaternions. Ce prolongement coïncide bien avec l'action par isométries pour la métrique hyperbolique, on identifie de cette manière $\PSL(2,\C)\rtimes \Z/2\Z$ avec $\Isom(\Hyp^3)$.\par
Si $\g$ est positif, on détermine son type en caculant sa trace. Si $\g$ est négatif, on détermine son type en calculant la trace de son carré, donnée par~:
$$\Tr(\g^2)=|a|^2+|d|^2+2\mathsf{Re}(b\bar{c})\quad \mathrm{si}\quad \g(z)= \frac{a\bar{z}+b}{c\bar{z}+d}\quad\mathrm{avec}\ a,b,c,d\in\C \ \mathrm{tels\ que}\ ad-bc=1.$$
L'égalité $|a|^2+|d|^2=(|a|-|d|)^2+2|ad|$ associée à l'inégalité triangulaire $|ad|-|bc|\geq-1$ entraîne $\Tr(\g^2)\geq -2$. En particulier, si $\g\in\PSL(2,\C)\rtimes \Z/2\Z$ est loxodromique négatif, alors $\Tr(\g^2)$ est positive car de module plus grand que $2$.
Soulignons que pour un élément $(A,-1)$ de $\PSL(2,\C)\rtimes \Z/2\Z$ le module de la trace n'est pas invariant par conjugaison.

\subsection{Une formule pour la distance de translation}\label{sec:translation}
Soit $\g$ un élément loxodromique de $\PSL(2,\C)$. Nous désignerons respectivement par $\ell_\g$ et $\theta_\g$ la \emph{distance de translation} et l'\emph{angle de rotation} de $\g$ le long de son axe. Nous conviendrons que $\theta_\g$ est défini par rapport à l'orientation canonique de $\C$ au point fixe répulsif de $\g$.\par
 Ces caractéristiques géométriques s'expriment facilement en fonction des valeurs propres. Si $\l_\g$ est la valeur propre de module plus grand que $1$ de $\g$, alors 
 $$\l_\g = \pm\exp\left(\frac{\ell_\g+i\theta_\g}{2} \right)\ \mathrm{ou\ de\ mani\grave ere\ \acute equivalente}\
\left\{\begin{array}{ccl} \ell_\g & = & 2\ln |\l_\g|\\ \theta_\g & = & 2\arg(\l_\g)\mod 2\pi\end{array}\right. .$$\par
 Ainsi $\ell_\g$ s'exprime en fonction de la trace de $\g$, puisque les valeurs propres sont racines du polynôme caractéristique. Malheureusement la formule obtenue s'avère difficile à manipuler. Il semble plus naturel de travailler avec $2\cosh(\ell_\g/2)$ qui s'exprime simplement en fonction de la trace. Le lemme ci-dessous généralise la formule classique $2 \cosh(\ell_\g/2)=|\Tr(\g) |$ valable pour tout élément hyperbolique $\g$ de $\PSL(2,\C)$. Dans le cas positif, cette généralisation se déduit d'une formule obtenue par D.~Gabai, R.~Meyerhoff et P.~Milley dans la preuve du corollaire~$3.6$ de \cite{gabai}. Afin de préserver l'unité du texte, nous reproduisons leur raisonnement dans la démonstration du lemme.

\begin{lemma}\label{lem:ellipse}
La distance de translation d'un élément loxodromique $\g$ de $\PSL(2,\C)\rtimes \Z/2\Z$ est donnée par~:
$$2\cosh\left(\frac{\ell_\g}{2}\right) =\left\{ \begin{array}{cl}
  \left|\frac{\Tr(\g)}{2}-1 \right|+ \left| \frac{\Tr(\g)}{2} +1\right|  & si\ \g\ \mathit{est\ positif,}\\
   & \\
 \sqrt{\Tr(\g^2)+2} & si\ \g\ \mathit{est\ n\acute egatif.}
\end{array}\right.$$
\end{lemma}

\begin{proof}
Cas $\g$ positif. Fixons un réel $u$ et considérons $\mathcal{E}_u\subset\C$ défini par
$$\mathcal{E}_u =  \{\cosh(u+iv)~;~ v\in\R \}.$$
Grâce à l'identité $\cosh(u+iv)=\cosh(u)\cos(v)+i\sinh(u)\sin(v)$, nous identifions $\mathcal{E}_u$ avec l'ensemble des points $(X,Y)$ de $\R^2$ satisfaisant l'équation $(X/\cosh u)^2+(Y/\sinh u)^2=1$. Le sous-ensemble $\mathcal{E}_u$ est donc une ellipse de $\C\simeq\R^2$ de foyers $1$ et  $-1$, ainsi
$$2\cosh(u) =|\cosh(u+iv)-1|+|\cosh(u+iv)+1|\quad \forall v\in\R. $$
L'expression souhaitée vient en prenant $u+iv=(\ell_\g+i\theta_\g)/2$.\par
Cas $\g$ négatif. La transformation $\g^2\in\PSL(2,\C)$ étant de type hyperbolique, on exprime sa distance de translation par la formule
$2\cosh(\g^2/2)=\Tr(\g^2). $
On conclut en remarquant que $\ell_{\g^2}=2\ell_\g$ et en utilisant une identité trigonométrique.
\end{proof}

\subsection{Cadre général}\label{sec:cadre-general}
 Considérons $M$ une $3$-variété hyperbolique non compacte mais de volume fini. Soit $\G$ un groupe uniformisant $M$, et soit $C$ un représentant maximal d'une cuspide de $M$ que nous relevons à $\Hyp^3$ en un empilement d'horoboules. Quitte à conjuguer $\G$ nous supposons que deux horoboules tangentes de l'empilement sont respectivement centrées en $0$ et $\infty$. Nous notons $B_x$ l'horoboule de l'empilement centrée en un point $x\in\partial\Hyp^3$.\par
 Soit $\g$ une isométrie envoyant $B_0$ sur $B_\infty$, nous avons déjà vu une telle isométrie au \textsection~\ref{sec:facteur2}. L'horoboule $B_b$ image de $B_\infty$ par $\g$ est tangente à $B_\infty$, mais n'appartient pas à l'orbite de $B_0$ sous l'action de $\G_\infty$. Les conditions $\g(0)=\infty$ et $\g(\infty)=b$ entraînent $\g=(A,\pm 1)$ avec
 $$ A=\left(\begin{array}{cc} cb & - c^{-1} \\ c & 0  \end{array}\right) \ \mathrm{pour\ un\ certain}\ c\in\C^\ast.$$
Lorsque $\g=(A,-1)$ nous avons $\g^2=(A\bar{A},1)$ avec
$$A\bar{A}=\left(\begin{array}{cc} |bc|^2-\bar{c}/c & -bc/\bar{c} \\ \bar{b}|c|^2 & -c/\bar{c} \end{array}\right).$$
Le point de tangence entre $B_0$ et $B_\infty$ étant envoyé sur le point de tangence entre $B_b$ et $B_\infty$, nous avons $c=e^{i\theta}/h$ avec $\theta\in\R$ et $h>0$ tel que $B_\infty=\{x_3>h\}.$ Ainsi $\g$ s'écrit
$$\g:z\mapsto b- \frac{h^2e^{-2i\theta }}{z}\quad \mathrm{ou}\quad \g:z\mapsto b- \frac{h^2e^{-2i\theta }}{\bar{z}}.$$
En particulier, suivant que $\g$ est positif ou négatif nous trouvons
$$\Tr(\g)=bc\quad \mathrm{ou}\quad \Tr(\g^2)=\frac{|b|^2}{h^2}-2\cos(2\theta).$$
Notez que si $\g$ est de type parabolique négatif, alors nécessairement $\Tr(\g^2)=2$.\par
 
\subsection{\'Eléments loxodromiques}\label{sec:majo-loxodromique}
 Parmi les éléments $\g\in\G$ envoyant $B_0$ sur $B_\infty$, ceux minimisant $|b|$ sont des candidats naturels pour réaliser la systole de $M$. Dans le cas où ils sont loxodromiques, nous contrôlons leur distance de translation à l'aide du lemme suivant~:
 
\begin{lemma}\label{lem:majoration-loxodromique}
Si $\g$ est un élément loxodromique de $\G$ envoyant $B_0$ sur $B_\infty$, alors
\begin{eqnarray*}
2\cosh(\ell_\g/2) & \leq & \frac{\sqrt{4h^2+|b|^2}}{h}.
\end{eqnarray*}
\end{lemma}

\begin{proof}
Si $\g$ préserve l'orientation, alors
$$2\cosh(\ell_\g/2) =  \left| \frac{bc}{2}-1\right| + \left| \frac{bc}{2}+1\right| \leq  \frac{\sqrt{4h^2+|b|^2}}{h}.$$
Le dernier terme vient de ce que le petit axe de l'ellipse $\mathcal{E}_u$ du lemme~\ref{lem:ellipse} est vertical. Si $\g$ renverse l'orientation, en combinant le lemme~\ref{lem:ellipse} avec l'expression de $\Tr(\g^2)$ nous trouvons 
$$2\cosh(\ell_\g/2)=\sqrt{\Tr(\g^2)+2}=\frac{\sqrt{|b|^2+2h^2(1-\cos\theta)}}{h}\leq  \frac{\sqrt{4h^2+|b|^2}}{h}.$$
\end{proof}

\subsection{\'Eléments paraboliques positifs}\label{sec:rigidite}
Un élément $\tau\in\G_\infty$ agit par translation, ou par translation-réflexion, sur l'horosphère $\partial B_\infty$ munie de sa métrique euclidienne induite. On définit sa \emph{distance de translation horosphérique} $\|\tau\|_\G$ comme la distance minimale sur $\partial B_\infty$ entre un point $p\in\partial B_\infty$ et son image $\tau(p)$. Remarquez que $\|\tau\|_\G$ est invariant par conjugaison par un élément du normalisateur de $\G$ dans $\Isom(\Hyp^3)$. Tout élément $\tau\in\L_\infty$ vérifie la minoration $\|\tau\|_\G\geq 1$ car il envoie $B_0$ sur une horoboule tangente à $B_\infty$ disjointe de $B_0$.\par
 Ceci nous amène à considérer le \emph{tour de taille} (\emph{waist size} en anglais) de la cuspide $\mathcal{C}$~:
$$\mathsf{w}(\mathcal{C})=\min_{\tau\in\G_\infty\setminus\{\mathsf{id}\}} \|\tau\|_\G.$$ 
Il s'agit de la systole de la surface plate $\G_\infty\backslash\partial B_\infty$. Cet invariant a été introduit par C.~Adams dans \cite{adams02}. Le résultat principal de son article donne une caractérisation du complément du n\oe ud de huit en terme de tour de taille.

\begin{theoremnonumber}[C.~Adams]
Soient $M$ une $3$-variété hyperbolique orientable, et $\cusp$ une cuspide de $M$. Alors $\mathsf{w}(\cusp)\geq 1$ avec égalité si et seulement si $M$ est isométrique au complément du n\oe ud de huit. 
\end{theoremnonumber} 

 Toute la difficulté du théorème réside dans la condition nécessaire du cas d'égalité. La proposition suivante repose sur cette condition nécessaire.

\begin{proposition}\label{pro:rigidite}
Si $M$ a une seule cuspide, et s'il existe un élément parabolique positif $\g$ de $\G$ envoyant $B_0$ sur $B_\infty$, alors $M$ est isométrique au complément du n\oe ud de huit ou à la variété de Gieseking.
\end{proposition}

\begin{proof}
Le point fixe de $\g$ est $b/2$. En particulier il y a une horoboule $B_{b/2}$ de l'empilement posée en $b/2$. Comme $\g$ envoie $0$ sur $\infty$ et $\infty$ sur $b$, nous voyons clairement que $\g$ stabilise le plan de $\Hyp^3$ engendré par $0$, $\infty$ et $b$. La trace de ce plan sur l'horosphère $\partial B_{b_/2}$ est une droite conforme de $\partial B_{b/2}$ elle aussi stabilisée par $\g$.\par
 \'Evaluons la distance de translation $\|\g\|_\G$ de $\g$ sur l'horosphère $\partial B_{b/2}$ munie de la métrique induite. Elle coïncide avec la distance entre les points d'intersection de $\partial B_{b/2}$ avec les droites hyperboliques $(0b/2)$ et $(\infty b/2)$. Nous pourrions expliciter cette distance en fonction du diamètre euclidien $d$ de $B_{b/2}$. Nous nous contenterons de deux remarques évidentes~: cette distance croît en fonction de $d$, et décroît en fonction de $|b|$. Ces grandeurs étant soumises aux contraintes d'empilement $d\leq h$ et $|b|\geq h$, on conclut que la valeur maximale de $\|\g\|_\G$ est atteinte lorsque les horoboules $B_0$, $B_\infty$ et $B_{b/2}$ sont mutuellement tangentes, d'où $\|\g\|_\G\leq 1$. Si $M$ est orientable on conclut par le théorème d'Adams.\par
 Si $M$ est non orientable, on considère son revêtement double orientable $M^+$. Rappelons que $M^+\simeq \G^+\backslash\Hyp^3$ où $\G^+$ désigne le sous-groupe d'indice deux formé des éléments positifs de $\G$. Ce revêtement double a une ou deux cuspides qui fibrent en tores. Les représentants maximaux des cuspides de $M^+$ coïncident avec les relevés à $M^+$ du représentant maximal de $\cusp$ ($\g$ est positif). Nous en déduisons que le tour de taille des cuspides de $M^+$ vaut $1$. En conclusion, $M^+$ est isométrique au complément du n\oe ud de huit, et $M$ est isométrique à la variété de Gieseking (unique $3$-variété hyperbolique non compacte de volume $\nu_3$).
\end{proof}

\subsection{\'Eléments paraboliques négatifs}\label{sec:para-negatif}
Supposons qu'il existe un élément parabolique négatif $\g\in\G$ envoyant $B_0$ sur $B_\infty$. Selon la discussion du \textsection~\ref{sec:cadre-general} nous avons $\Tr(\g^2)=2$ et
\begin{eqnarray*}
|b| & = 2h|\cos\theta|.
\end{eqnarray*} 
Nous déterminons par un calcul direct le point fixe $P_\g$ de $\g$ et $\g^2$~:
\begin{eqnarray*}
P_\g & = & \frac{b}{2} +ib\frac{h^2}{ |b|^2} \sin(2\theta) ,\\
 & = & \frac{b}{2}+i\frac{b}{2}\tan\theta,\\
 & = & \frac{b}{2}\cdot\frac{e^{i\theta}}{\cos\theta}.
\end{eqnarray*}
Le point $P_\g$ appartient à la médiatrice de $[0,b]$, et se trouve à une distance $h$ de $0$ et  de $b$.\par
 La distance de translation de $\g^2$ sur l'horosphère $\partial B_{P_\g}$ coïncide avec la distance entre les points d'intersection $P_0$ et $P_b$ des droites hyperboliques $(0P_\g)$ et $(bP_\g)$ avec $\partial B_{P_\g}$. Nous supposons dans un premier temps le diamètre euclidien de $B_{P_\g}$ égal à $h$. Les points $P_0$ et $P_b$ sont alors les points de tangence de $B_{P_\g}$ avec $B_0$ et $B_b$. Les horoboules $B_\infty$ et $B_{P_\g}$ jouent des rôles symétriques. En appliquant l'inversion suivant le cercle centré en $P_\g$ passant par $0$ et $b$, on voit que la distance recherchée vaut $|b|/h$, soit $2|\cos\theta|$. Supposons maintenant le diamètre $d$ de $B_{P_\g}$ quelconque. La distance hyperbolique entre $\partial B_{P_\g}$ et l'horosphère de diamètre $h$ vaut $|\ln(h/d)|$. En envoyant $P_\g$ sur $\infty$ on voit facilement que 
\begin{eqnarray*}
\|\g^2\|_\G & = & 2\frac{d}{h}|\cos\theta|.
\end{eqnarray*}
Les horoboules $B_0$ et $B_b$ étant disjointes nous avons $|b|\geq h$, avec $|b|=2h|\cos\theta|$ cela donne
\begin{eqnarray*}
|\cos\theta| & \geq &  1/2.
\end{eqnarray*}
Pour finir, remarquons que dans le cas où $M$  a une seule cuspide, l'horoboule $B_{P_\g}$ fait partie de l'empilement au-dessus de $C$, ce qui implique $h\geq d$.

\section{Une proposition sur les surfaces plates}\label{sec:surfaces-plates}

 Nous verrons au \textsection~\ref{sec:loxodromique} comment l'étude du rapport $\cosh(\sys/2)/\vol$ se ramène à un problème d'empilement de disques dans des surfaces plates. Ne disposant d'aucun résultat antérieur relatif à ce problème d'empilement (qui présente en soi peu d'intérêt), nous allons consacrer cette partie à la démonstration de la proposition ci-dessous, qui permettra d'étabir la proposition~\ref{pro:loxodromique}. Dans toute cette partie, par \emph{surface} nous entendons surface fermée, et par \emph{disque} nous entendons disque métrique ouvert plongé.
 
\begin{proposition}\label{pro:surfaces-plates}
Soit $N$ une surface fermée munie d'une métrique plate. Si $N$ contient deux disques disjoints de diamètre $h$, dont les centres sont à distance $d$ l'un de l'autre. Alors
$$\frac{h\sqrt{4h^2+d^2}}{\vol(N)}\leq \frac{\sqrt5}{\sqrt3}.$$
avec égalité ssi l'empilement de disques se relève en l'empilement hexagonal.
\end{proposition}

\begin{proof}
D'après le lemme~\ref{lem:realisation}, il existe une configuration de deux disques disjoints de même diamètre dans une surface plate réalisant la borne supérieure de $h\sqrt{4h^2+d^2}/\vol$. Selon le lemme~\ref{lem:d=h},  cette configuration satisfait $h\sqrt{4h^2+d^2}/\vol= \sqrt{5} h^2/\vol$. Or, nous savons depuis A.~Thue que la densité de deux disques dans une surface plate est majorée par $\pi/\sqrt{12}$ avec égalité ssi l'empilement se relève en l'empilement hexagonal (voir \cite{rogers,boroczky}), par suite $\vol\geq \sqrt{3} h^2$ avec égalité ssi l'empilement se relève en l'empilement hexagonal.
\end{proof}

\subsection{Réalisation de la borne supérieure}

\begin{lemma}\label{lem:realisation}
Il existe une surface plate, et une configurations de deux disques disjoints de même diamètre dans cette surface, réalisant la borne supérieure de $h\sqrt{4h^2+d^2}/\vol$.
\end{lemma}

\begin{proof}
Fixons un type topologique $S$ de surface fermée plate. Par compacité, toute métrique plate sur $S$ admet une configuration optimale relativement à $h\sqrt{4h^2+d^2}/\vol$. On définit de cette façon une fonction continue sur l'espace des modules des structures plates sur $S$.
 Regardons cet espace des modules comme l'espace des classes d'isométrie de métriques plates sur $S$ dont la plus courte géodésique orientable est de longueur $1$. Avec cette normalisation, on va à l'infini dans l'espace des modules (\emph{i.e.} on sort de tout compact) si et seulement si le volume tend vers l'infini.
La condition sur la plus courte géodésique orientable impose $h\leq 1$. Et, en travaillant avec des domaines fondamentaux rectangulaires, on trouve que le diamètre est inférieur à $\sqrt{1+\vol^2}$ quelle que soit la métrique plate sur $S$. Par conséquent $d\leq\sqrt{1+\vol^2}$, et $h\sqrt{4h^2+d^2}/\vol\lesssim 1$ lorsque le volume tend vers l'infini.
\end{proof}

\subsection{Rayon d'injectivité des bouteilles de Klein}\label{sec:klein}
 Soit $\Kl$ une bouteille de Klein plate. Le groupe fondamental, vu comme groupe des automorphismes du revêtement universel, admet un système de générateurs $(\a,\b)$ tel que~: $\a$ est une translation-réflexion, $\b$ est une translation, les directions de $\a^2$ et $\b$ sont orthogonales. Un tel système sera appelé une \emph{base orthogonale} de $\pi_1(\Kl)$. \'Etant fixée une base orthogonale $(\a ,\b)$, le groupe fondamental admet la présentation $\pi_1(\Kl)=\langle\a,\b~;\a\b=\b^{-1}\a \rangle$. En particulier, tout élément de $\pi_1(\Kl)$ s'écrit sous la forme normale $\a^k\b^l$ avec $k,l\in\Z$. À inverse et conjugaison près, il y a deux translation-réflexions primitives (une pour chaque plan projectif), correspondant à $\a$ et $\a\b$. Nous avons représenté en figure~\ref{fig:klein} les axes des translation-réflexions, ainsi que certaines droites dirigées suivant le vecteur $\b$~; les zones grisées mettent en évidence deux domaines fondamentaux.\par
 Le couple $(\a^2,\b)$ forme une base orthogonale du réseau euclidien d'indice deux de $\pi_1(\Kl)$. Une systole de $\Kl$ appartient à la classe d'homotopie libre de $\a$, $\a\b$ ou $\b$. Une géodésique orientable de longueur minimale appartient à la classe d'homotopie libre de $\a^2$ ou $\beta$.\par
 Bien que $\Kl$ ne soit pas homogène, le groupe à $1$-paramètre $t\mapsto t\a^2$ passe au quotient en une action par isométries de $\Sph^1$ sur $\Kl$. Parmi les géodésiques dans la classe d'homotopie libre de $\a$ (resp. $\a\b$), il y en a exactement une de longueur minimale, nous la noterons encore $\a$ (resp. $\a\b$). La découpe des géodésiques $\a$ et $\a\b$ produit un cylindre plat dont les bords sont de longueur $\|\a^2\|=2\|\a\|$, tandis que la distance séparant les bords est $\|\b\|/2$.\par

\begin{figure}[h]
\psfrag{a3}{$\a\b^{-1}$}\psfrag{a4}{$\a\b^{-2}$}\psfrag{E}{$\Eu^2$}
\psfrag{a}{$\a$}\psfrag{a1}{$\a\b$}\psfrag{a2}{$\a\b^2$}\psfrag{lb}{$\|\beta\|$}\psfrag{la}{$\|\a\|$}
\includegraphics[totalheight=5cm]{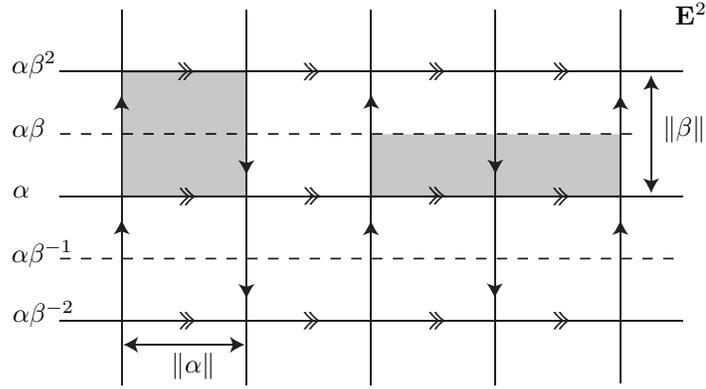}
\caption{Revêtement universel de $\Kl$}\label{fig:klein}
\end{figure}
 
 En tout point $p$ de $\Kl$, le rayon d'injectivité est réalisée par une géodésique fermée simple, librement homotope à $\a$, $\a\b$, $\a^2$ ou $\b$. Si $p$ est à distance $y$ de $\a$ (resp. $\a\b$), alors la géodésique passant par $p$ dans la classe d'homotopie libre de $\a$ (resp. $\a\b$) est de longueur $\sqrt{\|\a\|^2+4y^2}$. En particulier, les points sur la géodésique équidistante de $\a$ et $\a\b$ réalisent le maximum du rayon d'injectivité.\par

\subsection{Condition $d=h$}

\begin{lemma}\label{lem:d=h}
Si une surface plate, et une configuration de deux disques disjoints de même diamètre dans cette surface, réalisent le maximum du rapport $h\sqrt{4h^2+d^2}/\vol$, alors $d=h$. 
\end{lemma}

\begin{proof}
Considérons une configuration optimale, et supposons par l'absurde $d>h$. La configuration étant optimale, nous ne pouvons pas augmenter le diamètre $h$ des disques. Ainsi, l'adhérence d'un des disques contient un lacet géodésique non contractile. Si la surface est un tore, l'adhérence de chaque disque contient exactement une géodésique fermée simple non triviale. Nous avons unicité car les disques sont disjoints. Si la surface est une bouteille de Klein, plusieurs cas se présentent (nous reprenons les remarques du \textsection~\ref{sec:klein})~:
\begin{enumerate}[i)]
\item l'adhérence de chaque disque contient une géodésique fermée appartenant à $\b$~;
\item l'adhérence de chaque disque contient une géodésique fermée appartenant à $\a^2$~;
\item l'adhérence d'un disque contient un lacet géodésique librement homotope à $\a$~;
\item l'adhérence d'un disque contient un lacet géodésique librement homotope à $\a\b$~;
\end{enumerate}
Les deux disques étant disjoints, nous avons au plus un disque contenant un lacet homotope à $\a$ (resp. à $\a\b$). Pour la même raison, le cas i) exclut tous les autres cas, et les cas iii) et iv) ne peuvent se produire simultanément pour un même disque.\par
  Appelons $K$ la réunion des lacets géodésiques de longueur $h$ contenus dans l'adhérence des disques et passant par un des centres. En traitant séparément les différents cas (cas du tore, cas i-iv) de la bouteille de Klein), nous montrons sans difficulté qu'il existe une bande totalement géodésique $\mathcal{B}_\e$ de largeur $\e$ disjointe de $K$. Les géodésiques bordant $\mathcal{B}_\e$ sont parallèles aux géodésique fermées contenues dans l'adhérence des disques, en particulier elles sont de longueur inférieure à $2h$. La distance entre les centres des disques est réalisée par un segment géodésique de longueur $d$. Soient $u$ la distance parcouru par ce segment dans la direction définie par les bords de $\mathcal{B}_\e$, et $v$ la distance parcourue par ce segment dans la direction orthogonale, nous avons $d=\sqrt{u^2+v^2}$.\par
 Procédons à la chirurgie suivante~: nous découpons les géodésiques bordant $\mathcal{B}_\e$, nous enlevons $\mathcal{B}_\e$, et nous recollons les bords restants de manière à préserver l'alignement des géodésiques orthogonales à $\partial\mathcal{B}_\e$. Pour $\e$ suffisamment petit, cette chirurgie produit une nouvelle configuration de deux disques disjoints dans une surface plate (car $d>h$ et  $\mathcal{B}_\e$ n'intersecte pas $K$). Nous avons représenté en figure~\ref{fig:chirurgie} un exemple de chirurgie sur un tore plat, notez que la chirurgie change la pente des côtés non horizontaux du parallélogramme.\par
 
\begin{figure}[h]
\psfrag{b}{$b$}\psfrag{b1}{$b_\e$}\psfrag{u}{$u$}\psfrag{a}{$a$}\psfrag{f}{$\e$}
\psfrag{v}{$v$}\psfrag{v1}{$v_\e$}\psfrag{d}{$d$}\psfrag{d1}{$d_\e$}\psfrag{e}{$\mathcal{B}_\e$}
\includegraphics[totalheight=6cm]{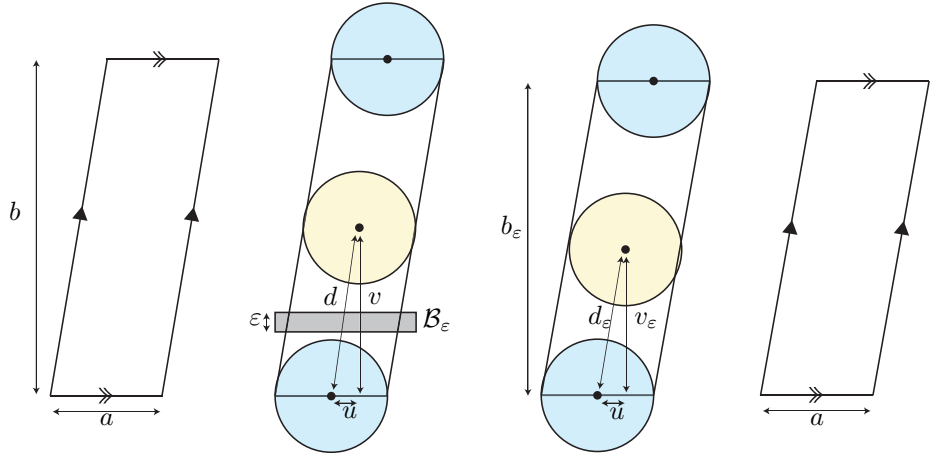}
\caption{Chirurgie sur un tore plat}\label{fig:chirurgie}
\end{figure}
 
 En marquant d'un indice $\e$ les grandeurs relatives à la nouvelle configuration nous avons~:
$h_\e=h$, $u_\e=u$, $v_\e=v-\e$, $d_\e^2=d^2-2v\e+\e^2.$ Le calcul de $d_\e$ dépend du choix du segment réalisant la distance entre les centres des disques, nous prenons celui dont la longueur décroît le plus par chirurgie. 
La surface fibre en géodésiques parallèles aux bords de $\mathcal{B_\e}$, et son volume s'écrit sous la forme $\vol=ab$, avec $a$ et $b$ satisfaisant~: $a\geq h$ et $b\geq v$. Nous avons $\vol_\e=\vol\cdot (1-\e/b)$ et
\begin{eqnarray*}
\frac{h_\e\sqrt{4h_\e^2+d_\e^2}}{\vol_\e} & = & \frac{h\sqrt{4h^2+d^2}}{\vol} \left(1+\frac{\e}{b} -\frac{v\e}{4h^2+d^2}+o(\e) \right).
\end{eqnarray*}
Ou bien $4h^2+d^2>b v$, alors pour $\e>0$ suffisamment petit nous avons
\begin{eqnarray*}
\frac{h_\e\sqrt{4h_\e^2+d_\e^2}}{\vol_\e} & > & \frac{h\sqrt{4h^2+d^2}}{\vol}.
\end{eqnarray*}
Ou bien $4h^2+d^2\leq b v$, alors en utilisant $h\leq a$ et $v\leq b$ nous trouvons $h\sqrt{4h^2+d^2}/\vol\leq 1$.
Dans les deux cas la configuration n'est pas maximale, confirmant l'hypothèse d'absurde.
\end{proof}

\section{Une inégalité systolique optimale en dimension $3$}\label{sec:ineg-optimale}
 Dans cette partie nous prouvons~:
\begin{theorem}
Toute $3$-variété hyperbolique non compacte satisfait
\begin{eqnarray*}
\frac{\cosh(\sys/2)}{\vol_\triangle} & \leq & \frac{\sqrt{5} }{2} ,
\end{eqnarray*}
sauf si elle est isométrique à la variété de Gieseking, auquel cas
\begin{eqnarray*}
\frac{\cosh(\sys/2)}{\vol_\triangle} & = & \frac{1+\sqrt{13} }{4} .
\end{eqnarray*}
\end{theorem} 
 
 Soit $M$ une $3$-variété hyperbolique non compacte mais de volume fini. Considérons un groupe $\G$ uniformisant $M$, et $C$ un représentant maximal de la cuspide de plus petit volume de $M$. Nous relevons $C$ à $\Hyp^3$ en un empilement d'horoboules. Quitte à conjuguer $\G$ nous supposons que deux horoboules tangentes de l'empilement sont respectivement centrées en $0$ et $\infty$. Nous posons $B_\infty=\{x_n>h\}$ avec $h >0$.\par
 Pour contrôler le rapport $\cosh(\sys/2)/\vol_\triangle$, nous devons minorer le volume simplicial et majorer la sytole. À l'aide du théorème de Meyerhoff-Kellerhals et de l'égalité \eqref{eq:volume1} nous obtenons une première minoration du volume simplicial~: 
$$\vol_\triangle\geq \frac{2}{\sqrt{3}}\cdot \vol(\cusp)= \frac{\covol(\G_\infty)}{\sqrt{3}h^2}.$$
Le covolume de $\G_\infty$ sera minoré par une méthode \emph{ad hoc}, ou par un argument classique de densité qui donne $\covol(\G_\infty)\geq \sqrt{3}h^2$ (voir \textsection\textsection~\ref{sec:facteur2} et \ref{sec:minoration-volume}). 
Nous majorons la systole par la distance de translation d'un élément de $\G$ envoyant $B_0$ sur $B_\infty$. Parmi les éléments envoyant $B_0$ sur $B_\infty$, nous appelons $\g$ celui minimisant le module de $b=\g(\infty)$ parmi les points de $\G_\infty\cdot b$ ($\g$ n'est pas toujours unique). Lorsque cet élément est loxodromique (resp. parabolique positif) nous concluons rapidement grâce à la proposition~\ref{pro:surfaces-plates} (resp. grâce à la proposition~\ref{pro:rigidite}). Le cas parabolique négatif demande un peu plus de travail.\par
 D'une manière générale nous reprenons les notations introduites au \textsection~\ref{sec:cadre-general}. Toutefois nous travaillons avec la normalisation suivante~: $\eta:z\mapsto z+1$ \emph{est un vecteur minimal du réseau euclidien} $\L_\infty$. Remarquez que cela implique $\Real(b)\in[-1/2,1/2]$, $h\leq1$, $\covol(\L_\infty)\geq \sqrt{3}/2$ et $\covol(\G_\infty)\geq 1/2$.

\subsection{Si $\g$ est loxodromique}\label{sec:loxodromique}

\begin{proposition}\label{pro:loxodromique}
Si $\g$ est loxodromique, alors 
\begin{eqnarray*}
\frac{\cosh(\ell_\g/2)}{\vol_\triangle(M)} & \leq & \frac{\sqrt{5}}{2}.
\end{eqnarray*}
\end{proposition}

\begin{proof} Le lemme~\ref{lem:majoration-loxodromique} associé à l'inégalité $\vol(M)\geq \covol(\G_\infty)/\sqrt{3}h^2$ donne
\begin{eqnarray*}
\frac{\cosh(\ell_\g/2)}{\vol_\triangle(M)} & \leq & \frac{\sqrt{3}}{2}\cdot\frac{h\sqrt{4h^2+|b|^2}}{\covol(\G_\infty)}.
\end{eqnarray*}
Les horoboules des orbites $\G_\infty\cdot B_0$ et $\G_\infty\cdot B_b$ se projettent orthogonalement sur $\partial B_\infty$ pour former un empilement de disques de diamètre $h$. Cet empilement passe au quotient en un empilement de deux disques de diamètre $h$ dans $\G_\infty\backslash \Eu^2$. Comme $b$ est supposé de module minimal parmi les point de $\G_\infty\cdot b$, les centres des deux disques sont à distance $|b|$ l'un de l'autre. On conclut en appliquant la proposition~\ref{pro:surfaces-plates}.
\end{proof}

\subsection{Si $\g$ est parabolique positif}\label{sec:para-positif}

\begin{proposition}
Si $\g$ est parabolique positif, ou bien $M$ a plusieurs cuspides et 
\begin{eqnarray*}
\frac{\cosh(\ell_{\eta^{\pm1}\g}/2)}{\vol_\triangle(M)} & \leq & \frac{\sqrt{5}}{2},
\end{eqnarray*}
ou bien $M$ est isométrique à la variété de Gieseking ou au complément du n\oe ud de huit.
\end{proposition}

\begin{proof} Nous supposons $\g$ parabolique positif. Si $M$ a une seule cuspide, nous savons grâce à la proposition~\ref{pro:rigidite} que $M$ est isométrique à la variété de Gieseking ou au complément du n\oe ud de huit. Nous supposons désormais que $M$ possède plusieurs cuspides.\par
Puisque $\g$ est parabolique positif, nous avons $\Tr(\g)=bc=\pm 2$. Nécessairement l'une des transformation $\eta^{\pm1}\g$ est loxodromique, et en appliquant le lemme~\ref{lem:ellipse} nous avons~:
\begin{eqnarray*}
\cosh(\ell_{\eta^{\pm1}\g}/2) & \leq & 1 +\frac{1}{2h}.
\end{eqnarray*}
Soit $\cusp'$ une cuspide de $M$ distincte de $\cusp$, et soit $C'$ son représentant de plus grand volume parmi ceux disjoints de $C$. Une analyse des points de tangence entre $C$ et $C'$ montre que
$$\left\{\begin{array}{lccl} \vol(C)\geq\sqrt{3}/2 & \mathrm{et} & \vol(C')\geq\sqrt{3}/2, & \mathrm{ou} \\
 \vol(C)\geq 3\sqrt{3}/4 & \mathrm{et} & \vol(C')\geq\sqrt{3}/4. & 
 \end{array}\right.$$
Dans tous les cas nous avons $\vol_\triangle(M)\geq \frac{2}{\sqrt{3}}\cdot \vol(C\cup C') \geq 2$. On trouve ce raisonnement et ce résultat dans l'article \cite{adams88} de C.~Adams. Cette minoration avec le majorant ci-dessus permet de conclure dès que $h\geq 1/2(\sqrt{5}-1)$. En utilisant les minorations
$$\left\{\begin{array}{ccccc} \vol(C) & \geq & \covol(\G_\infty)/2h^2 & \geq & 1/4h^2 \\
 \vol(C') & \geq & \sqrt{3}/4, &  
 \end{array}\right.$$
nous trouvons $\vol_\triangle(M)\geq \frac{2}{\sqrt{3}}\cdot \vol(C\cup C') \geq \frac{1}{2}\cdot \left(1+\frac{1}{\sqrt{3}h^2}\right)$. Cette minoration avec le majorant ci-dessus nous donne une fonction croissante de $h$, qui prend une valeur inférieure à $\sqrt{5}/2$ en $h=1/2(\sqrt{5}-1)$.
\end{proof}

\subsection{Si $\g$ est parabolique négatif}

\subsubsection{Unicité de la cuspide et conséquences}

\begin{lemma}\label{lem:final1}
Si $\g$ est parabolique négatif, et si $M$ a plusieurs cuspides, alors
\begin{eqnarray*}
\frac{\cosh(\ell_{\eta^{\pm1}\g}/2)}{\vol_\triangle(M)} &\leq& \frac{\sqrt{5}}{2}.
\end{eqnarray*}
\end{lemma}

\begin{proof} Supposons $\g$ parabolique négatif, nous avons $|b|=2h|\cos\theta|$ vu le \textsection~\ref{sec:para-negatif}. L'une des transformations $\eta^{\pm 1}\g$ est loxodromique, et par le lemme~\ref{lem:ellipse} il vient
\begin{eqnarray*}
2\cosh(\ell_{\eta^{\pm1}\g}/2) & = &\frac{1}{h} \sqrt{|b\pm1|^2+4h^2\sin^2 \theta}\\
 & = & \frac{1}{h} \sqrt{|b|^2\pm 2\Real(b)+1+4h^2\sin^2 \theta}\\
 & \leq & \frac{1}{h} \sqrt{4h^2+2}.
\end{eqnarray*}
Supposons que $M$ possède plusieurs cuspides. Nous reprenons les minorations du volume simplicial vues dans la preuve de la proposition précédente. Si $h\geq 1/2\sqrt{2}$, en associant la minoration $\vol_\triangle(M)\geq 2$ avec la majoration ci-dessus nous obtenons l'inégalité souhaitée. Si $h\leq 1/2\sqrt{2}$, nous concluons en utilisant la deuxième minoration.
\end{proof}

À partir de maintenant nous supposons que $M$ \emph{a une seule cuspide}. Dans ce cas, \emph{il existe un élément $\a\in\G_\infty$ conjugué à $\g$ dans $\G$}.

\begin{lemma}
Si $\g$ est parabolique négatif, et si $M$ a une seule cuspide, alors $\a^2$ et $\eta$ sont égaux à inverse près.
\end{lemma}

\begin{proof}
 D'après le \textsection~\ref{sec:para-negatif} nous avons $\|\a\|_\G=\|\g \|_\G=\frac{d}{h} |\cos\theta|$. Ainsi, le vecteur de translation de $\a$ (vue comme isométrie de $\Eu^2$) est de norme $\|\a\|_\Eu=d|\cos\theta|\leq h\leq 1$. Ceci implique que $\a$ est un élément primitif de $\G_\infty$ ($\a$ n'est pas une puissance non triviale d'un élément de $\G_\infty$), car les vecteurs minimaux de $\L_\infty$ sont de norme $1$.\par
 Les vecteurs $\a^2$ et $\eta$ étant des éléments primitifs de $\L_\infty$, leur colinéarité implique leur égalité à inverse près. Aussi nous supposons par l'absurde $\a^2$ orthogonal à $\eta$. Le point $b$ minimise le module parmi les points de son orbite sous l'action de $\L_\infty$, d'où $|\Real(b)|\leq 1/2$ et $|\Ima(b)|\leq d|\cos\theta|$. Nous en déduisons
\begin{eqnarray*}
|b|^2 & \leq & d^2\cos^2\theta+\frac{1}{4}, \\
4h^2\cos^2\theta & \leq & h^2\cos^2\theta+\frac{1}{4}, \\
2h|\cos\theta|& \leq & \frac{1}{\sqrt{3}}.
\end{eqnarray*}  
Finalement $\|\a^2\|_\Eu=2d |\cos\theta|\leq 1/\sqrt{3}$, ce qui contredit l'hypothèse $\eta$ vecteur minimal.
\end{proof}

 Nous supposons que $\g$ est parabolique négatif, et que $M$ a une seule cuspide. Par le lemme ci-dessus nous avons $\|\a^2\|_\Eu=\|\eta\|_\Eu=1$, ce qui entraîne (voir  \textsection~\ref{sec:para-negatif})
$$d=\frac{1}{2|\cos\theta|}\quad \mathrm{et}\quad |b|=\frac{h}{d}.$$
Comme $|\cos\theta|\geq 1/2$ et $h\geq d$ (voir \textsection~\ref{sec:para-negatif}), il vient
$$h\geq d \geq \frac{1}{2}\quad \mathrm{et}\quad 2h\geq |b| \geq 1.$$

\subsubsection{Dichotomie}

Nous poursuivons avec les mêmes hypothèses ($\g$ est parabolique négatif et $M$ a une seule cuspide). Nous distinguons deux cas suivant que le point fixe de $\g$ (noté $P_\g$) appartient ou non à l'une des orbites $\G_\infty\cdot 0$ ou $\G_\infty\cdot b$.

\begin{lemma}
Si $P_\g$ appartient à l'une des orbites $\G_\infty\cdot 0$ ou $\G_\infty\cdot b$, alors $M$ est isométrique à la variété de Gieseking.
\end{lemma}

\begin{proof}
Supposons que $P_\g$ appartient à l'une des orbites, disons à $\G_\infty\cdot 0$. Alors, l'horoboule $B_{P_\g}$ est de même diamètre que $B_0$, soit $d=h$. Nous en déduisons d'une part que les horoboules $B_0$ et $B_b$ sont tangentes à $B_{P_\g}$ (voir \textsection~\ref{sec:para-negatif}), et d'autre part que $|b|=h/d=1$.
Par hypothèse $b$ et $0$ réalisent la distance entre les orbites $\G_\infty\cdot 0$ et $\G_\infty\cdot b$. Ainsi nous avons
$1=|b|\leq |b-P_\g|=h\leq 1.$
 Finalement $h=1$ et la distance de translation horosphérique de $\a^2$ vaut $\|\a^2\|_\G=1$. Comme $M$ a une seule cuspide qui fibre en bouteilles de Klein, $M^+$ a une seule cuspide qui fibre en tores. Ainsi $\|\a^2\|_{\G^+}=1$ et par le théorème d'Adams (\textsection~\ref{sec:para-negatif}) $M^+$ est isométrique au complément du n\oe ud de huit, donc $M$ est isométrique à la variété de Gieseking.
\end{proof}

\begin{lemma}
Si $P_\g$ n'appartient à l'une des orbites $\G_\infty\cdot 0$ ou $\G_\infty\cdot b$, alors
\begin{eqnarray*}
\frac{\cosh(\ell_{\eta^{\pm1}\g}/2)}{\vol_\triangle(M)} &\leq& \frac{\sqrt{5}}{2}.
\end{eqnarray*}
\end{lemma}

\begin{proof}
Nous connaissons déjà la majoration $\cosh(\ell_{\eta^{\pm1}\g}/2)\leq \sqrt{1+1/2h^2}$ (voir la preuve du lemme~\ref{lem:final1}). Il s'agit donc de minorer $\vol_\triangle(M)\geq\covol(\G_\infty)/\sqrt{3}h^2$. Afin d'estimer le covolume de $\G_\infty$, nous allons regarder comment certaines horoboules se projettent dans la bouteille de Klen plate $\G_\infty\backslash\Eu^2$. Rappelons qu'une horoboule se projette orthogonalement sur un disque de $\Eu^2$, qui passe au quotient en un disque de $\G_\infty\backslash\Eu^2$.\par
 L'isométrie parabolique négative $\g$ envoie l'horoboule $B_0$ sur une horoboule qui lui est tangente (en l'occurence $B_\infty$). L'isométrie $\a$ se comporte de la même manière puisqu'elle est conjuguée à $\g$ dans $\G$~: elle envoie une horoboule $B_P$ sur une horoboule qui lui est tangente. Comme $B_0$ est à distance $\ln(h/d)$ de $B_{P_\g}$, il vient que $B_P$ est de diamètre $d$.\par
 Les horoboules $B_0$ et $B_b$ se projettent sur deux disques $D_0$ et $D_b$ de diamètre $h$ dans la bouteille de Klein plate $\G_\infty\backslash\Eu^2$. Les centres des disques sont à distance $|b|\geq 1$ l'un de l'autre. L'horoboule $B_P$ se projette sur un disque $D_P$ de diamètre $d\geq1/2$ dans $\G_\infty\backslash\Eu^2$. Comme $\a$ envoie $B_P$ sur une horoboule qui lui est tangente, le centre de $D_P$ est à distance minimale de la géodésique $\a$. Dans la suite nous supposons $|b|=1$, $d=1/2$ et le centre de $D_P$ supporté par $\a$. Ceci ne pose pas de problème car nous minorons le volume de $\G_\infty\backslash\Eu^2$.\par
\begin{figure}[h]
\psfrag{b}{$D_b$}\psfrag{A}{$1/2$}\psfrag{P}{$D_P$}\psfrag{O}{$D_0$}\psfrag{B}{$\sqrt{h/2-1/16}$}
\psfrag{C}{$\sqrt{3}/2$}\psfrag{D}{$\sqrt{h^2-1/4}$}\psfrag{a}{$\a^2$}\psfrag{h}{$h$}
\begin{minipage}[b]{.48\linewidth}
\centering\epsfig{figure=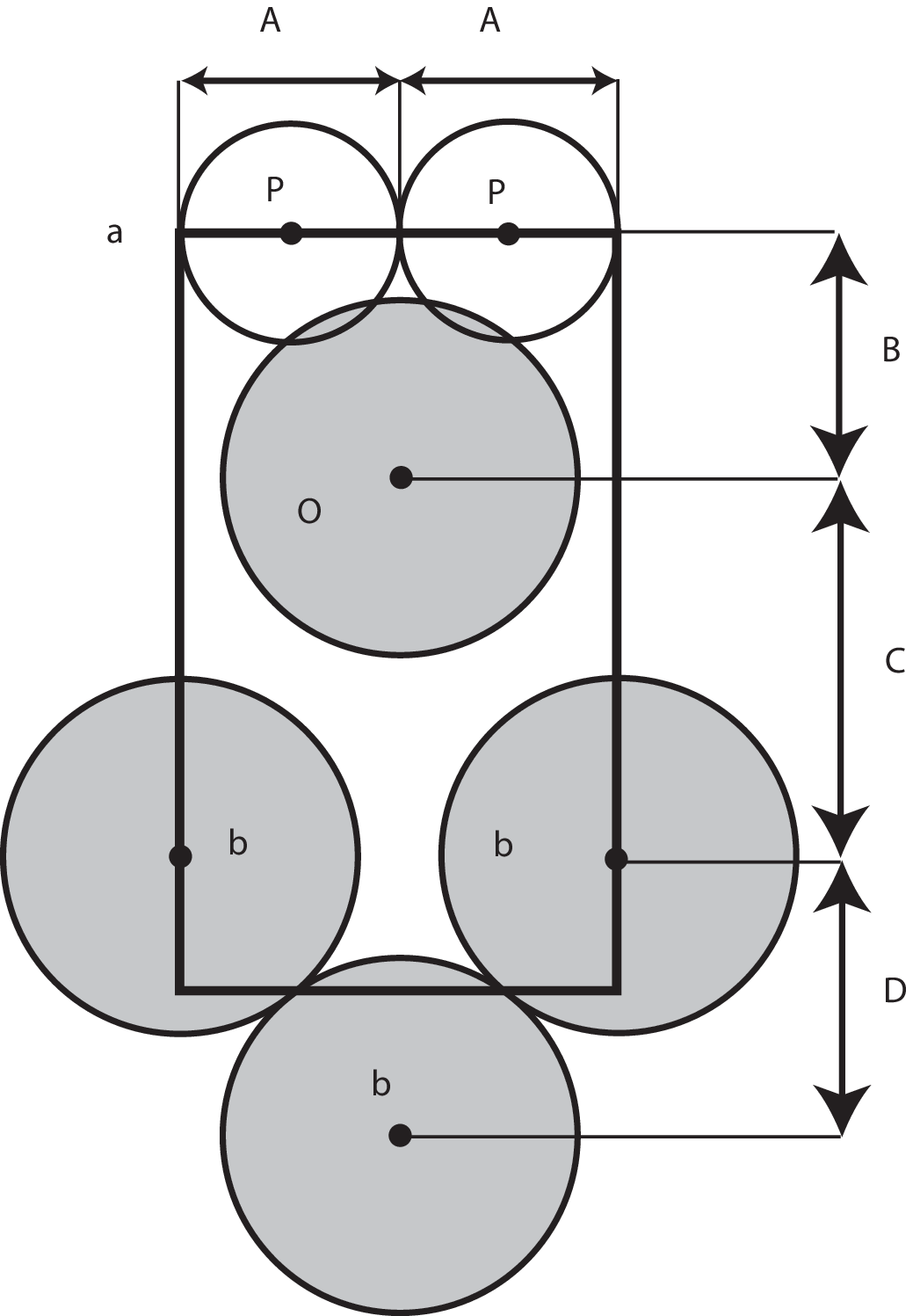,totalheight=5.5cm}
\caption{Configuration optimale}\label{fig:configuration}
\end{minipage}
\begin{minipage}[b]{.48\linewidth}
\centering\epsfig{figure=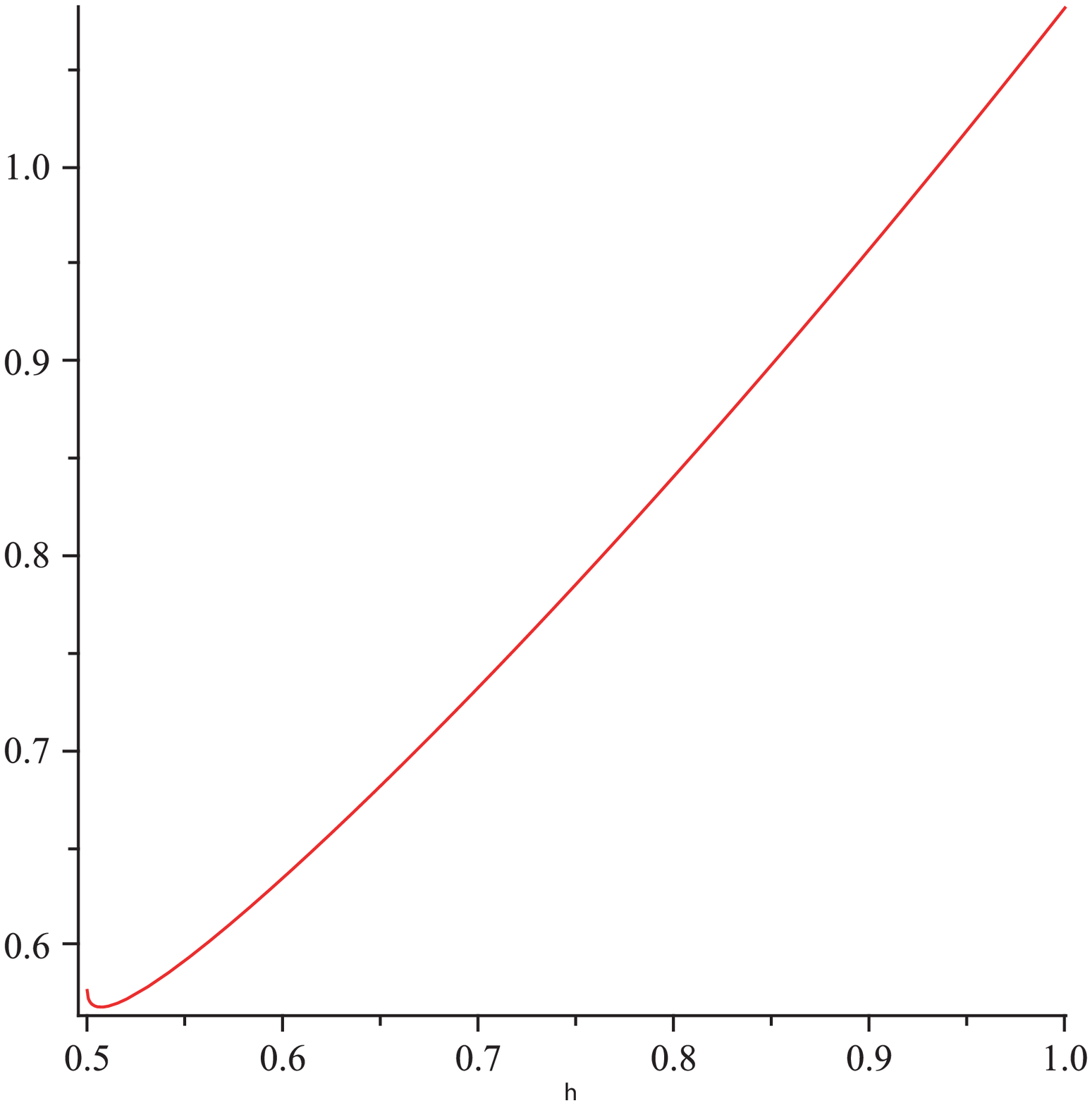,totalheight=5.5cm}
\caption{Graphe du majorant}\label{fig:graphe}
\end{minipage}
\end{figure}
 Nous regardons la bouteille de Klein comme un cylindre dont chaque bord est autorecollé par antipodie. Les deux bords s'identifient aux géodésiques $\a^2$ et $(\a\b)^2$ qui sont de longueur $1$. Les disques $D_0$ et $D_b$ sont clairement disjoints, en revanche $D_P$ peut intersecter un de ces disques. Si les horoboules $B_P$ et $B_0$ sont tangentes, alors les centres de $D_P$ et $D_0$ sont à distance $\sqrt{h/2}$ l'un de l'autre, en particulier les disques s'intersectent. Nous avons représenté en figure~\ref{fig:configuration} la configuration minimisant le volume de $\G_\infty\backslash\Eu^2$. Le rectangle est un domaine fondamental pour l'action de $\G_\infty$ sur $\Eu^2$, les côtés horizontaux correspondent à $\a^2$ et $(\a\b)^2$, c'est-à-dire aux bords du cylindre évoqué précédemment. On trouve
\begin{eqnarray*}\covol(\G_\infty) & \geq & \sqrt{h/2-1/16}+\sqrt{3}/2+\sqrt{h^2/4-1/16},\end{eqnarray*}
et 
\begin{eqnarray*} \frac{\cosh(\ell_{\eta^{\pm 1}\g}/2)}{\vol_\triangle(M)} & \leq &\frac{ h \sqrt{3h^2+3/2}}{ \sqrt{h/2-1/16}+\sqrt{3}/2+\sqrt{h^2/4-1/16}}.\end{eqnarray*}
Une étude (fastidieuse) des variations du majorant montrerait qu'il atteint son maximum en $h=1$ sur l'intervalle $[1/2,1]$, et que ce maximum est inférieur à $\sqrt{5}/2$. Nous préférons donné une représentation graphique de ce majorant (figure~\ref{fig:graphe}).
\end{proof}

\subsection{La systole de la variété de Gieseking}\label{sec:systole-gieseking}
En utilisant l'algorithme de C.~Hodgson et J.~Weeks (\cite{hodgson} \textsection~$3$), le logiciel SnapPea trouve l'approximation $\sys(\mathrm{N}1_1)\simeq 1.087$. Or, depuis la thèse de H.~Gieseking (\cite{gieseking} p.~$185$), on sait construire un sous-groupe $\G$ de $\PGL(2,\mathcal{O}_3)\rtimes \Z/2\Z$ uniformisant N$1_1$. Ceci justifie $\sys(\mathrm{N}1_1)=2\mathrm{arccosh}((1+\sqrt{13})/4)$. Nous rappelons ci-dessous la construction de $\G$, et déterminons des éléments réalisant la systole.\par
 Partons du tétraèdre idéal régulier $T=(01\omega\infty)$ avec $\omega=e^{i\pi/3}$. Nous construisons sans difficulté deux homographies identifiant les faces de $T$ comme indiqué en figure~\ref{fig:Gieseking}. Voici les expressions analytiques et matricielles de ces homographies (en fait ce ne sont pas des matrices mais des éléments de $\mathsf{PGL}(2,\C)\rtimes \Z/2\Z$, d'où l'indice $-1$)~:
\begin{eqnarray*}
f:z\mapsto \frac{\bar{z}-1}{-\omega}, & F=\left(\begin{array}{cc}
1 & -1 \\
0 & -\omega 
\end{array}\right)_{-1},& \infty \omega 1\mapsto \infty 1 0,\\
g:z\mapsto \frac{\omega\bar{z}}{\bar{z}+\omega},&
G=\left(\begin{array}{cc}
\omega & 0 \\
1 & \omega 
\end{array}\right)_{-1},&  0\omega\infty\mapsto 01\omega.
\end{eqnarray*}
Ces transformations engendrent un sous-groupe $\G=\langle f,g\rangle$ de $\PGL(2,\mathcal{O}_3)\rtimes \Z/2\Z$ où $\mathcal{O}_3$ désigne l'anneau des entiers de $\Q(\sqrt{-3})$. Le théorème de Poincaré donne la présentation $\G=\langle f,g~|~g^{-1}f^{-1}g^2f^2 \rangle$. Le quotient de $\Hyp^3$ par $\PGL(2,\mathcal{O}_3)\rtimes \Z/2\Z$ (resp. $\PGL(2,\mathcal{O}_3)$) est la $3$-orbivariété hyperbolique non compacte (resp. non compacte orientable) de plus petit volume (R.~Meyerhoff \cite{meyerhoff}). Le groupe $\PGL(2,\mathcal{O}_3)\rtimes \Z/2\Z$ est associé à un orthoschème de Coxeter, en figure~\ref{fig:decomposition1} nous avions décomposé le tétraèdre idéal régulier en $6$ tétraèdres, chacun de ces tétraèdres se décompose à nouveau en $4$ orthoschèmes de Coxeter, montrant ainsi que l'indice de $\G$ dans $\PGL(2,\mathcal{O}_3)\rtimes \Z/2\Z$ est égal à $24$.\par
Les translation-réflexions $f$ et $(fg^2)^{-1}g(fg^2)$ engendrent le stabilisateur $\G_\infty$. Elles ont pour axes les droites $1/2+\R\omega$ et $-3/2+\R\omega$, et pour vecteurs de translation $-\omega/2$ et $\omega/2$. Les horoboules tangentes à $B_\infty$ sont de diamètre $1$, elles se projettent en un empilement de disques de diamètre $1$ sur $\C$. Cet empilement se divise en deux orbites sous l'action de $\G_\infty$~: $\G_\infty\cdot B_0$ et $\G_\infty\cdot B_{\omega^2}$.
Les deux horoboules de $\G_\infty\cdot B_{\omega^2}$ tangentes à $B_0$ sont $B_{\omega^2}$ et $B_{-1}$. Notons $\tau_\omega$ la translation suivant le vecteur $\omega$. La transformation $g^{-1}\tau_\omega$ envoie $B_0$ sur $B_\infty$ et $B_\infty$ sur $B_{\omega^2}$, elle correspond à l'élément $\g$ étudié lors des paragraphes précédents (nous aurions pu tout aussi bien considérer $\tau_{-\omega}g^{-1}\tau_\omega$). Cet élément est parabolique négatif (de même que $\tau_{-\omega}g^{-1}\tau_\omega$). Les horoboules de $\G_\infty\cdot B_{\omega^2}$ les plus proches de $B_0$ sans en être tangentes sont $B_{-1-\omega}$, $B_{\omega^2-1}$, $B_{i\sqrt{3}}$ et $B_{1-\omega}$. Les transformations envoyant $B_0$ sur $B_\infty$ et $B_\infty$ sur $B_{-1-\omega}$ ou $B_{i\sqrt{3}}$ sont loxodromiques négatives, leur distance de translation vaut $2\mathrm{arccosh}(\sqrt{3/2})$. Les transformations envoyant $B_0$ sur $B_\infty$ et $B_\infty$ sur $B_{\omega^2-1}$ ou $B_{1-\omega}$ sont loxodromiques positives, leur distance de translation vaut $2\mathrm{arccosh}((1+\sqrt{13})/4)$.\par
 Nous avons représenté en figure~\ref{fig:empilement} la projection sur $\C$ des horoboules tangentes à $B_\infty$. Les deux couleurs permettent de distinguer les deux orbites sous l'action de $\G_\infty$. Nous avons dessiné deux domaines fondamentaux~: l'un constitué de triangles équilatéraux, l'autre formé d'un rectangle marqué des identifications habituelles. 
\begin{figure}[h]
\psfrag{0}{$0$}\psfrag{1}{$1$}\psfrag{o}{$\omega$}\psfrag{oo}{$\omega^2$}\psfrag{R}{$\Real$}\psfrag{I}{$\Ima$}
\includegraphics[totalheight=5.5cm]{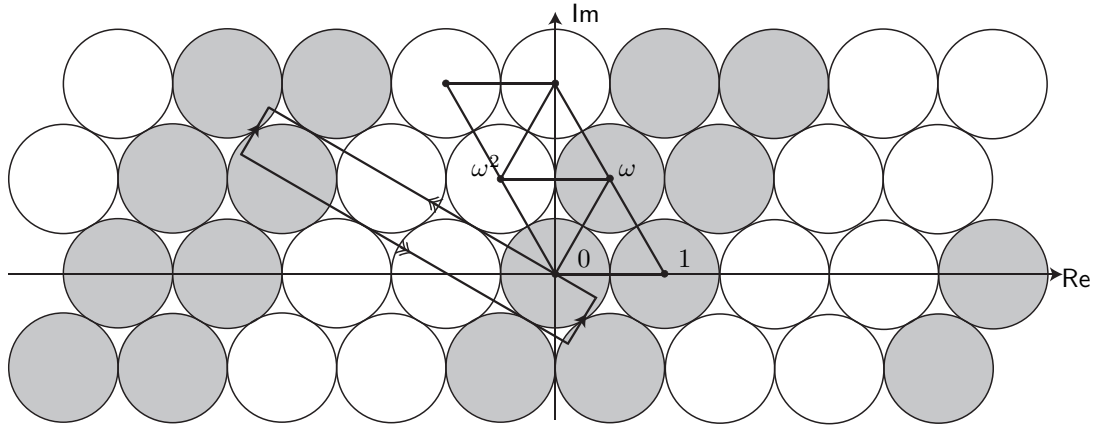}
\caption{Empilement d'horoboules pour N$1_1$}\label{fig:empilement}
\end{figure}


\bibliographystyle{alpha}
\bibliography{biblio}

\end{document}